\documentclass{amsart}
\usepackage{amsthm,amssymb,graphicx,xcolor,enumerate}
\usepackage{hyperref}

\newtheorem{proposition}{Proposition}[section]
\theoremstyle{remark}
\newtheorem{remark}{Remark}[section]
\newtheorem{algorithm}{Algorithm}[section]
\DeclareMathOperator{\diag}{diag}
\DeclareMathOperator{\Div}{div}
\DeclareMathOperator{\Id}{Id}

\author[P. Greengard]{Philip Greengard}
\address{Department of Statistics, Columbia University}

\author[J.G. Hoskins]{Jeremy G. Hoskins}
\address{Department of Statistics, University of Chicago}

\author[N.F. Marshall]{Nicholas F. Marshall}
\address{Program in Applied and Computational Mathematics and the Department of
Mathematics, Princeton University}

\author[A. Singer]{Amit Singer}
\address{Program in Applied and Computational Mathematics and the Department of
Mathematics, Princeton University}

\title[On a linearization of quadratic Wasserstein distance]{On a linearization
of quadratic Wasserstein distance}

\keywords{Wasserstein distance, Sobolev norm, Witten Laplacian}
\thanks{P.G. is supported by the Alfred P. Sloan Foundation.  N.F.M. is supported in
part by NSF DMS-1903015.
A.S. is supported in part by AFOSR FA9550-20-1-0266, the Simons Foundation Math+X
Investigator Award,  NSF BIGDATA Award IIS-1837992, NSF DMS-2009753, and NIH/NIGMS
1R01GM136780-01. }

\begin{document}

\begin{abstract}
This paper studies the problem of computing a linear approximation of
quadratic Wasserstein distance $W_2$. In particular, we compute
an approximation of the negative homogeneous weighted Sobolev norm
whose connection to Wasserstein distance follows from a classic
linearization of a general Monge-Amp\'ere equation. Our contribution is
threefold. First, we provide expository material on this classic linearization of Wasserstein distance including a quantitative error estimate. 
Second, we reduce the computational problem to solving a elliptic 
boundary value problem involving the Witten Laplacian, which is a 
Schr\"odinger operator of the form $H = -\Delta + V$, and describe an
associated embedding. Third, for the case of probability distributions on the
unit square $[0,1]^2$ represented by $n \times n$ arrays we present a fast code
demonstrating our approach. Several numerical examples are presented.
\end{abstract}

\maketitle

\section{Introduction}
\subsection{Introduction} \label{intro}
Let $\mu$ and $\nu$ be probability measures supported on a bounded convex
set $\Omega \subset \mathbb{R}^N$.
The quadratic Wasserstein distance $W_2(\mu,\nu)$ is defined by
\begin{equation} \label{defw2}
W_2(\mu,\nu)^2 := \inf_\pi \int_{\Omega \times \Omega} |x - y|^2 d\pi(x,y),
\end{equation}
where the infimum is taken over all transference plans $\pi$ (probability
measures on $\Omega \times \Omega$ such that $\pi[A \times \Omega] = \mu[A]$
and $\pi[\Omega \times A] = \nu[A]$ for all
measurable sets $A$). Computing the quadratic Wasserstein distance
is a nonlinear problem. In this paper, we consider the case where $\mu$ and $\nu$ 
have smooth positive densities $f$ and $g$ with respect to Lebesgue
measure: $d\mu= f dx$ and $d\nu = g\, dx$. In this case,
there is a classic local linearization of $W_2$
based on a weighted negative  homogeneous
Sobolev  norm, which is derived from  linearizing a general Monge--Amp\'ere
equation, see \S \ref{prelim}. In particular, the weighted negative homogeneous
Sobolev norm $\|\cdot\|_{\dot{H}^{-1}(d\mu)}$ is defined for functions $u$ such that $\int_\Omega u d\mu = 0$ by
\begin{equation} \label{Hdefn}
\|u\|_{\dot{H}^{-1}(d\mu)}  := \sup \left\{ \int_\Omega u\, \varphi\, d\mu :
\|\varphi\|_{\dot{H}^1(d\mu)} = 1 \right\},
\end{equation}
where
\begin{equation} \label{eqh1d}
\|\varphi\|_{\dot{H}^1(d\mu)}^2 := \int_\Omega |\nabla \varphi |^2 d\mu.
\end{equation}
That is, the space $\dot{H}^{-1}(d\mu)$ is the dual space of $\dot{H}^1(d\mu)$.
Under fairly general conditions, if $\delta \mu$ denotes a perturbation of
$\mu$, then (informally speaking) we have
$$
W_2(\mu,\mu + \delta\mu) = \|\delta\mu\|_{\dot{H}^{-1}(d\mu)} + o(\|\delta \mu\|),
$$
see for example [Theorem 7.2.6 \cite{Villani2003}] for a precise
statement. Under stronger assumptions, if $\|\delta \mu\| = \varepsilon$, then 
the error term can be shown to be $\mathcal{O}(\varepsilon^2)$, see \S
\ref{preciselin} for details.
In this
paper, we study how to leverage the connection between
the
$\dot{H}^{-1}(d\mu)$-norm and the $W_2$
metric for computational  purposes. In particular, 
our computational approach is based on a connection between  the
$\dot{H}^{-1}(d\mu)$-norm and the Witten Laplacian, which is a Schr\"odinger
operator of the form
$$
H = -\Delta + V,
$$
where $V$ is a potential that depends on $f$, see Figure \ref{fig01}. We show
that this connection provides a method of computation whose
computational cost can be controlled by the amount of regularization used when
defining the potential, see \S \ref{compsec} for details.

\begin{figure}[ht!]
\centering
\begin{tabular}{cc}
\includegraphics[width =.45\textwidth]{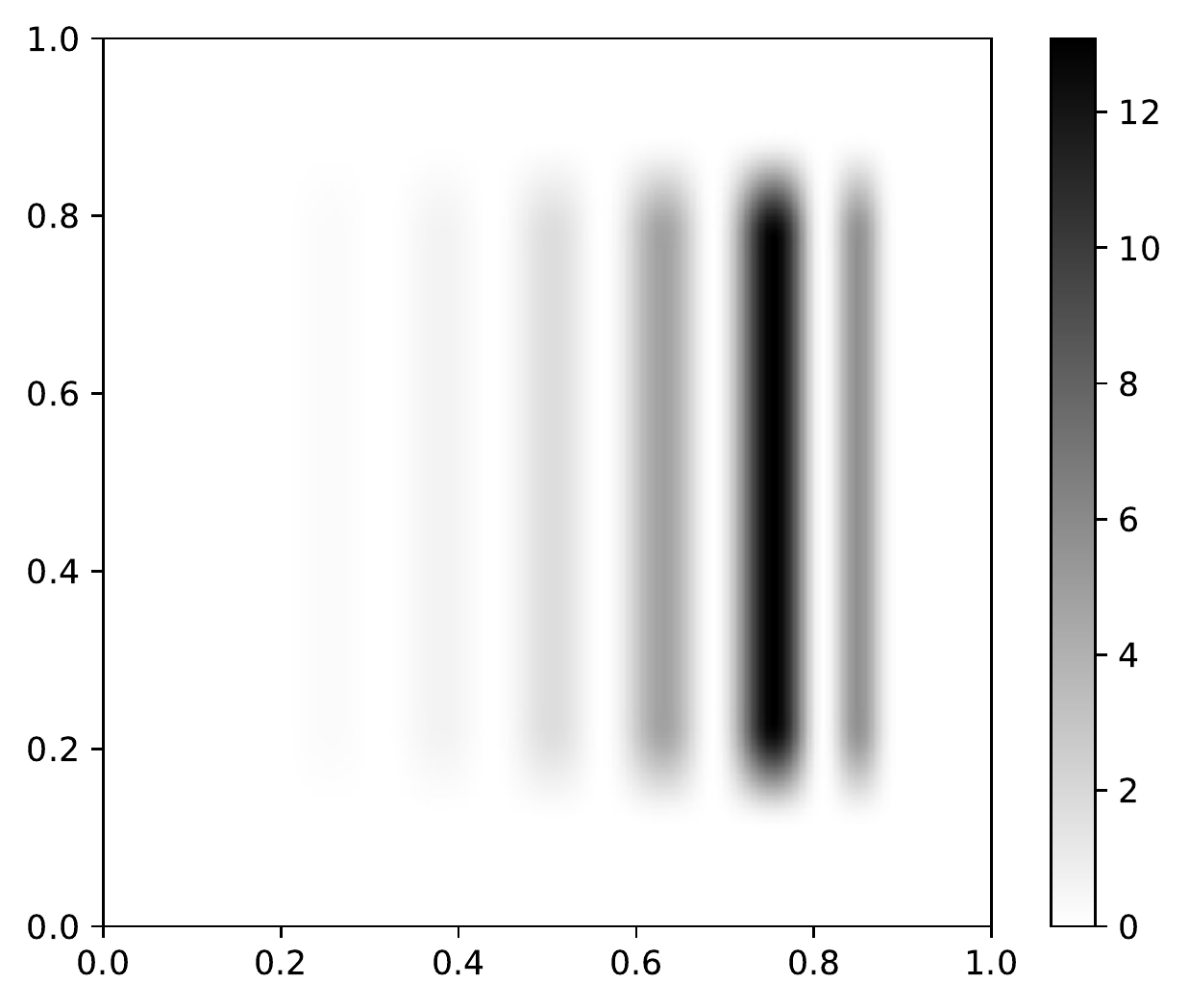}  &
\includegraphics[width =.45\textwidth]{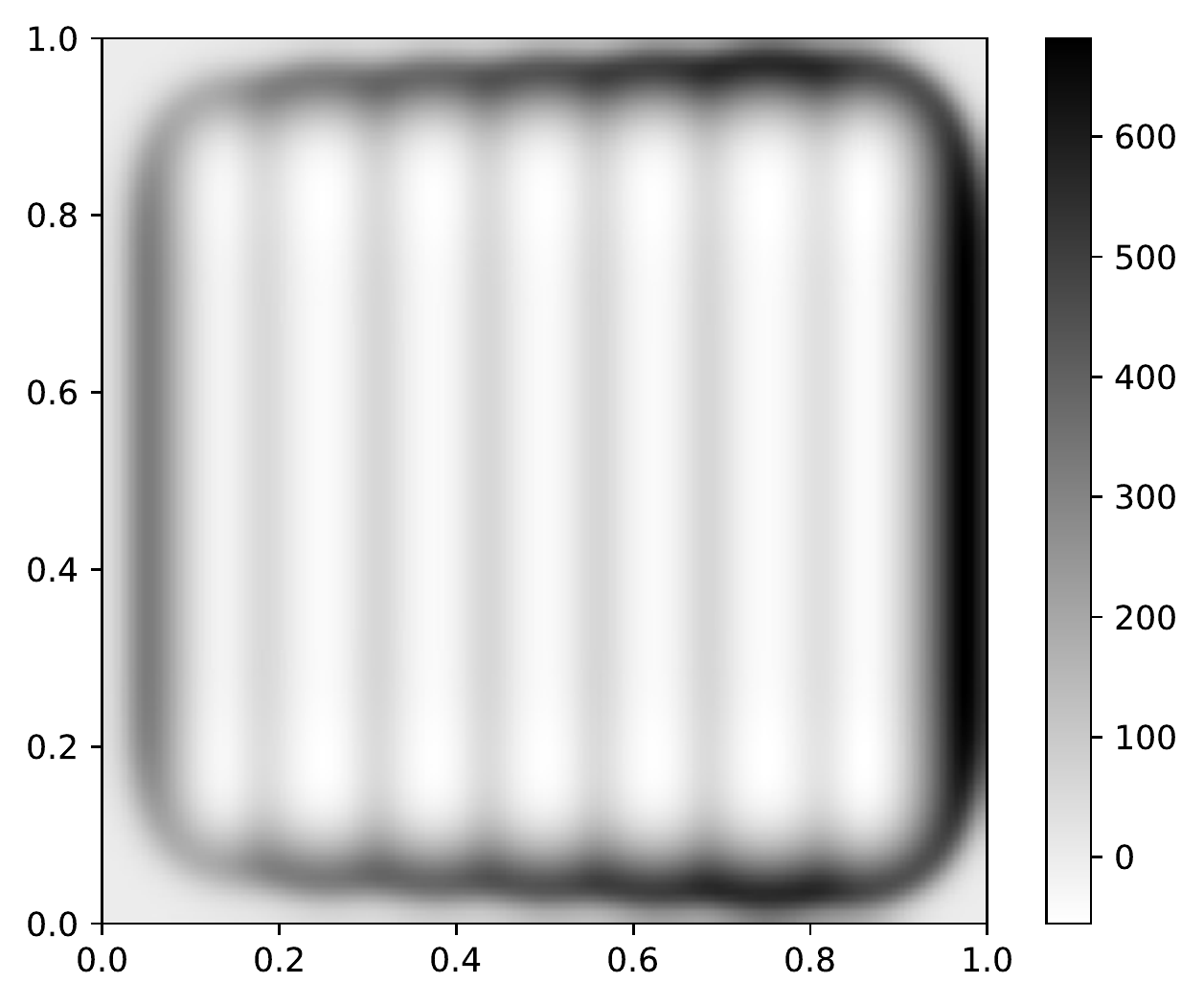}  
\end{tabular}
\caption{Function $f$ (left) and its regularized potential $V$
(right), see \S \ref{translateex} for details about this example.}
\label{fig01}
\end{figure}

For the case of probability distributions on the unit square $[0,1]^2$
represented by $n \times n$ arrays, we present a code for computing this
linearization of $W_2$ based on the Witten Laplacian, and present a number of
numerical examples, see \S \ref{numericssec}.  This Witten Laplacian
perspective leads to several potential applications; in particular, a method of
defining an embedding discussed and methods of smoothing
discussed in \S \ref{discuss}. 

\subsection{Background} \label{relatedwork}
The quadratic Wasserstein distance $W_2$ is an instance
of Monge --Kantorovich optimal transport whose study was initiated by 
Monge \cite{Monge1781} in 1781 and generalized by Kantorovich
\cite{Kantorovich1942} in 1942. The theory of optimal transport has
been developed by many authors; for a summary see
the book by Villani  \cite{Villani2003}. Recently, due to new applications in
data science and machine learning, developing methods to compute and approximate optimal 
transport distances has become an important area of research in applied
mathematics, see the surveys by Peyr\'e and  Cuturi \cite{PeyreCuturi2019} and
Santambrogio \cite{Santambrogio2015}.

In this paper, we focus on the connection between the $W_2$ metric and the
$\dot{H}^{-1}(d\mu)$-norm, which can be used to  approximate $W_2$, see \S
\ref{linear} and \S \ref{preciselin}. The connection between
$\dot{H}^{-1}(d\mu)$ and $W_2$ follows from the
work of Brenier \cite{Brenier1987} in 1987 who discovered that under
appropriate conditions the solution to Monge--Kantorovich optimal transport 
with a quadratic cost can be expressed using a map that pushes forward one
measure to the other; moreover, this map is the gradient of a convex
function, see \S \ref{mongekant}. Brenier's theorem reduces  the problem of
computing Wasserstein distance to solving a generalized Monge--Amp\'ere equation.
Linearizing this general Monge--Amp\'ere equation  gives rise to an elliptic
equation, which corresponds to a weighted negative homogeneous Sobolev norm, see
\S \ref{linear} below or see [\S 4.1.2, \S 7.6  of \cite{Villani2003}]. 
The connection between Wasserstein distance and negative homogeneous Sobolev
norms has been considered by  several different authors as is discussed in the
following section. 

\subsection{Related work} \label{related}
Several authors have considered the connection between negative homogeneous
Sobolev norms and the $W_2$ metric in several different contexts. In analysis,
this connection has been used to establish estimates by many authors, see for
example \cite{Karpukhin2021,Ledoux2020,  Peyre2016, Steinerberger2020,
Steinerberger2021}.
Moreover, this connection also arises in the study of  how measures change under
heat diffusion,  see \cite{Chen2021, Otto2000, vonRenesse2005}.  The weighted
negative homogeneous Sobolev norm has also been
considered in connection to maximum  mean discrepancy (MMD) which is a
technique that can be used to compute a distance between point clouds, and has
 many applications in machine learning, see \cite{ArbelKorbaSalimGretton2019,
MrouehSercuRaj2019, Nietert2021}. Authors have also considered this connection 
in papers focused on computing the $W_2$ metric, see \cite{Cancs2020,
EngquistRenYang2020, LebratGournayKahnWeiss2019}. Moreover, the negative
homogeneous Sobolev norm has been considered in several applications to seismic
image and image processing, see \cite{Dunlop2021, Engquist2021, Yang2021}.

We emphasize three related works. First, Peyre
\cite{Peyre2016}  establishes estimates for the $W_2$
metric in terms of the unweighted negative homogeneous Sobolev norm
$\|\cdot\|_{\dot{H}^{-1}(dx)}$. Let $\mu$ and $\nu$ be probability measures on
$\Omega$ with densities $f$ and $g$ with respect to Lebesgue measure.  The
result of Peyre says that if $0< a < f,g < b < +\infty$, then
$$
b^{-1/2} \| f -g \|_{\dot{H}^{-1}(dx)} \le W_2(\mu,\nu) \le
a^{-1/2} \|f -g  \|_{\dot{H}^{-1}(dx)};
$$
a discussion and concise proof of this result can be found in
[\S 5.5.2 of \cite{Santambrogio2015}]. Informally, this result says that
if $\mu$ and $\nu$ have their mass spread out over $\Omega$, then the
Wasserstein distance is equivalent to the negative unweighted homogeneous
Sobolev norm of $f - g$; as a consequence, the weighted Sobolev norm is most
interesting for probability distributions that have regions of high and low
density.

Second, Engquist, Ren, and Yang \cite{EngquistRenYang2020} study the application
of Wasserstein distance to inverse data matching. In particular, they compare
the effectiveness of several different Sobolev norms for their applications; we
note that the definitions of the norms they consider differ from the norm
\eqref{Hdefn} that we consider, see the discussion in \S \ref{fourier} below,
but their results do indicate that negative homogeneous Sobolev norms may not be
an appropriate substitute for the $W_2$ metric for some applications.

Third, Yang, Hu, and Lou \cite{Yang2021} consider the implicit regularization
effects of Sobolev norms in image processing. In \S \ref{discuss} we mention
a similar potential applications to image smoothing; our perspective is slightly
different, but the underlying idea of this potential application is the same as
\cite{Yang2021}.

\section{Preliminaries and motivation} \label{prelim}

In this section, we briefly summarize material from Chapters 0.1, 0.2, 2.1, 2.3,
4.1, 4.2, and 7.6 of the book by Villani \cite{Villani2003}. In order to make
these preliminaries as concise as possible we state all definitions and theorems
for our special case of interest; in particular, we assume that all probability
measures have smooth densities with respect to Lebesgue measure, and restrict
our attention to transport with respect to a quadratic cost function. This
section is organized as follows: we consider the Monge-Kantorovich transport
problem and Brenier's theorem in \S \ref{mongekant}, the Monge-Amp\'ere equation
in \S \ref{mongeamp}, and then discuss linearization of the $W_2$ metric in \S
\ref{linear}. The main purpose of these preliminaries is to provide background
for \eqref{wasss}, stated at the end of \S \ref{linear}, which clarifies the
statement that the $\dot{H}^{-1}(d\mu)$-norm is a linearization of the $W_2$
metric.

\subsection{Monge-Kantorovich optimal transport and Brenier's theorem} \label{mongekant}
Let $\Omega \subset \mathbb{R}^n$ be a bounded convex set, and $\mu$ and $\nu$ be
probability measures on $\Omega$ that have positive smooth densities $f$ and
$g$, respectively, with respect to the Lebesgue measure. A transference plan
$\pi$ is a
probability measure on $\Omega \times \Omega$ such that 
$$
\pi[A \times \Omega] = \mu[A],
\quad \text{and} \quad \pi[\Omega \times A] = \nu[A],
$$
for all measurable subsets $A$ of $\Omega$.  We denote the set of all
transference plans by $\Pi(\mu,\nu)$, and define the transportation cost
$I(\pi)$ with respect to the quadratic cost function $|x - y|^2$ by 
$$
I(\pi) = \int_{\Omega \times \Omega} |x - y|^2 d\pi(x,y).
$$
In this case, the Monge-Kantorovich optimal transport cost
$\mathcal{T}(\mu,\nu)$ is defined by
\begin{equation} \label{otc}
\mathcal{T}(\mu,\nu)  = \inf_{\pi \in \Pi(\mu,\nu)} I(\pi).
\end{equation}
Since we are considering a quadratic transportation cost, the
 Monge-Kantorovich optimal transport cost $\mathcal{T}(\mu,\nu)$ is the square of the $W_2$ metric:
\begin{equation} \label{w2t}
W_2(\mu,\nu)^2 := \mathcal{T}(\mu,\nu).
\end{equation}
Under the above conditions, Brenier's theorem states that the Monge-Kantorovich
optimization problem \eqref{otc} has a unique solution $\pi$ satisfying 
\begin{equation} \label{eq1}
d\pi(x,y) = d\mu(x)\, \delta( y = \nabla \varphi(x)),
\end{equation} 
where $\varphi$ is a convex function such that $\nabla \varphi \#
\mu = \nu$. Here $\#$ denotes the push-forward, and $\delta$ denotes a Dirac
distribution. To be clear, given $T : \Omega \rightarrow \Omega$,
the push-forward $\nu : = T_\# \mu$ is defined by the relation $\nu(A) =
\mu(T^{-1}(A))$ for all measurable sets $A$, or equivalently, by the relation
$$
\int_\Omega \zeta(x) d\nu(x) = \int_\Omega \zeta(T(x)) d\mu(x),
$$
for all continuous functions $\zeta$ on $\Omega$. Combining \eqref{otc}, \eqref{w2t}, and \eqref{eq1} gives 
\begin{equation} \label{T}
W_2(\mu,\nu)^2 = \int_\Omega |x - \nabla \varphi(x) |^2 d\mu(x).
\end{equation}

\subsection{Monge-Amp\'ere equation} \label{mongeamp}
Recall that by assumption $\mu$ and $\nu$ have densities $f$ and $g$, respectively, with
respect to the Lebesgue measure.  We can express \eqref{eq1} as
$$
\int_\Omega \zeta(\nabla \varphi(x)) f(x) dx =
\int_\Omega \zeta(y) g(y) dy ,
$$
for all bounded continuous functions $\zeta$ on $\Omega$. Changing variables $y
= \nabla \varphi(x)$ on the right hand side and using the fact that $\zeta$ is
arbitrary gives
\begin{equation} \label{eq2}
f(x) = g(\nabla \varphi(x)) \det D^2 \varphi(x).
\end{equation}
Upon rearranging this equation as 
$$
\det D^2 \varphi(x) = \frac{f(x)}{g(\nabla \varphi(x))},
$$
it is clear that is an instance of the general Monge-Amp\'ere equation 
$$
\det D^2 \varphi(x) = F(x,\varphi, \nabla \varphi),
$$
which has been studied by many authors, see for example \cite{Lions1985,
Lions1986, Trudinger2008}.
    
\subsection{Linearization of general Monge-Amp\'ere equation} \label{linear} To
linearize this general Monge-Amp\'ere equation we assume that $f > 0$ is positive, $\nabla \varphi
\approx \Id$ is close to the identity, and $f \approx g$. More precisely,
assume \begin{equation} \label{eq3}
\varphi(x) = \frac{|x|^2}{2} + \varepsilon \psi(x), 
\end{equation}
and 
\begin{equation} \label{eq3b}
g(x) = \big(1 + \varepsilon u(x) \big) f(x),
\end{equation}
for some $\varepsilon > 0$.
Substituting \eqref{eq3} into
\eqref{T} gives \begin{equation} \label{eqpsi}
 W_2(\mu,\nu)^2 = \varepsilon^2 \int_\Omega | \nabla
\psi |^2 d\mu,
\end{equation}
which expresses $W_2(\mu,\nu)^2$ in terms of $\psi$. Substituting 
\eqref{eq3} and \eqref{eq3b} into \eqref{eq2} gives
$$
f = (1+ \varepsilon u + \varepsilon^2 R_1)(f + \varepsilon \nabla f \cdot
\nabla \psi + \varepsilon^2 R_2)(
1 + \varepsilon \Delta \psi + \varepsilon^2 R_3),
$$
where $R_1,R_2,R_3$ are remainder functions depending on
$f,\psi$, $u$. It follows from rearranging terms that
\begin{equation}  \label{LeqR}
L \psi = u + \varepsilon R,
\end{equation}
where
$$
L \psi = - \Delta \psi - \nabla(\log f) \cdot \nabla \psi,
$$
and $R$ denotes some remainder function depending on $f$, $\psi$, and $u$. 
Thus, if $\psi$ satisfies $L \psi = u$, then by \eqref{eqpsi} and
\eqref{LeqR} we expect that
\begin{equation} \label{wasss}
W_2(\mu,\nu) =  \varepsilon \sqrt{\int_\Omega
|\nabla \varphi|^2
d\mu} + \mathcal{O}(\varepsilon^2),
\end{equation}
which, roughly speaking, says that  $L \psi = u$ is a linearization of the
quadratic Wasserstein optimal transport problem, see \S \ref{preciselin} for a
more precise version of \eqref{wasss}.

\section{Characterizing the operator \texorpdfstring{$L$}{L}}

So far we have presented background material that motivates why the quadratic
Wasserstein distance $W_2$ is related to the operator $L$ defined by 
$$
L \psi = - \Delta \psi - \nabla(\log f) \cdot \nabla \psi.
$$
In this section, we discuss the connection between the operator $L$, the
negative weighted homogeneous Sobolev norm $\|\cdot\|_{\dot{H}^{-1}(d\mu)}$, and
the quadratic Wasserstein distance $W_2$ in detail.  The section is organized as
follows: We start, in \S \ref{greens}, by stating and proving a
version of Green's first identity that $L$ satisfies. Second, in \S
\ref{hminus1} we state a result connecting the $\dot{H}^{-1}(d\mu)$-norm to
the solution of an elliptic boundary value problem involving $L$. Third, in \S
\ref{divf}, we give a characterization of the $\dot{H}^{-1}(d\mu)$-norm in terms
of a divergence optimization problem.  Fourth, in \S \ref{wittenf} we consider a
characterization of the $\dot{H}^{-1}(d\mu)$-norm involving the Witten
Laplacian, which can be derived from $L$ by a change of variables. Fifth, in \S
\ref{preciselin} we provide a more precise version of the statement that the
$\dot{H}^{-1}(d\mu)$-norm is a linearization of the $W_2$ metric. Finally, in \S
\ref{fourier}, we discuss weighted Sobolev norms in relation to the Fourier
transform.

\subsection{Green's first identity analog for \texorpdfstring{$L$}{L}} \label{greens}
Let $\Omega$ be a bounded convex domain in $\mathbb{R}^N$,
$\varphi$ be a once differentiable function, and $\psi$ be a
twice differentiable function. Green's first identity states that
\begin{equation} \label{green}
\int_\Omega \varphi (-\Delta \psi) dx = \int_\Omega (\nabla \varphi) \cdot
(\nabla \psi) dx 
-\int_{\partial \Omega} \varphi (\partial_n \psi) d s,
\end{equation}
where $\partial_n \psi = n \cdot \nabla \psi$, and $n$ is an exterior
unit normal to the surface element $ds$.
\begin{proposition} \label{proplg}
Suppose that $\mu$ is a measure on $\Omega$ which has a density $f$ with
respect to Lebesgue measure: $d\mu = f dx$. Further assume that $f$ is once
differentiable, $f > 0$ on $\Omega$, and $L := -\Delta  - \nabla( \log
f) \cdot \nabla$. Then,
\begin{equation} \label{Ldiv}
\int_\Omega \varphi\,\, (L \psi) d\mu = 
\int_\Omega (\nabla \varphi) \cdot (\nabla \psi) d\mu,
\end{equation}
whenever $\partial_n \psi = 0$ on $\partial \Omega$ or $\varphi = 0 $ on $\partial \Omega$.
\end{proposition}
\begin{proof}
By the definition of $L$ and the fact that $d\mu = f dx$ we have
$$
\int_\Omega \varphi (L \psi) d\mu = \int_\Omega \varphi f (-\Delta \psi) -
\varphi \,\,(\nabla f) \cdot (\nabla \psi) dx.
$$
Since we assumed $\partial_n \psi$ or $\varphi$ vanishes on $\partial \Omega$ it follows from
\eqref{green} that
$$
\int_\Omega \varphi f (-\Delta \psi) - \varphi \,\,(\nabla f) \cdot (\nabla \psi)
dx = \int_\Omega (\nabla (\varphi f)) \cdot (\nabla \psi) - \varphi\,\, (\nabla f)
\cdot (\nabla \psi) dx.
$$
Finally, observe that expanding $\nabla(\varphi f)$ with the product
rule and canceling terms gives:
$$
\int_\Omega  f \,\,(\nabla \varphi) \cdot (\nabla \psi) +
\varphi (\nabla f) \cdot (\nabla \psi) - \varphi\,\, (\nabla f)
\cdot (\nabla \psi) dx =  \int_\Omega (\nabla \varphi) \cdot (\nabla \psi) d\mu,
$$
which establishes \eqref{Ldiv}.
\end{proof}

\subsection{Elliptic boundary value problem} \label{hminus1}
Let $\Omega$ be a bounded convex domain in $\mathbb{R}^N$, and suppose that
$\mu$ is a measure on $\Omega$ which has a density $f$ with respect to Lebesgue
measure: $d\mu = f dx$. Assume that $f$ is once differentiable and $f > 0$ on
$\Omega$.  For functions $u$ such that $\int_\Omega u d\mu =0$ we define the
$\dot{H}^{-1}(d\mu)$-norm by
$$
\|u \|_{\dot{H}^{-1}(d\mu)} := \sup \left\{ \int_\Omega u\, \varphi\, d \mu :
\|\varphi \|_{\dot{H}^1(d\mu)} = 1 \right\},
$$
where
$$
\|\varphi\|_{\dot{H}^1(d\mu)}^2 := \int_\Omega |\nabla \varphi |^2 d\mu.
$$
The following result characterizes the $\dot{H}^{-1}(d\mu)$-norm in terms of an
elliptic boundary value problem involving the operator $L$.

\begin{proposition} \label{prop1}
We have
\begin{equation} \label{elliptich1}
\|u\|_{\dot{H}^{-1}}^2 = \int_\Omega |\nabla \psi |^2 d\mu = \int_\Omega \psi
(L \psi) d\mu  = \int_\Omega \psi\, u\,  d\mu ,
\end{equation}
where $\psi$ is a solution to the elliptic boundary value problem
\begin{equation} \label{elliptic}
\left\{
\begin{array}{cc}
L \psi = u & \text{in } \Omega \\
\partial_n \psi = 0 & \text{on } \partial \Omega ,
\end{array} \right.  
\quad
\text{where} \quad
L := - \Delta + \nabla(-\log f) \cdot \nabla.
\end{equation}
\end{proposition}
\begin{proof}[Proof of Proposition \ref{prop1}]
The second and third equalities in 
\eqref{elliptich1} are a direct consequence of the definition of $L$ 
and Proposition \ref{proplg}, so we only need to show that
$\|u\|_{\dot{H}^{-1}}^2 = \int_\Omega
|\nabla \psi |^2 d\mu$ . Assume that $\psi$ is a solution to \eqref{elliptic}
substituting $L \psi = u$. It follows from \eqref{Ldiv} that
\begin{equation} \label{uvarphi}
\int_\Omega u\, \varphi\, d \mu = \int_\Omega (L \psi)\, \varphi\, d\mu
= \int_\Omega (\nabla \psi) \cdot (\nabla \varphi)\, d\mu.
\end{equation}
Using the Cauchy-Schwarz inequality gives
$$
\int_\Omega (\nabla \psi) \cdot (\nabla \varphi) d\mu \le 
\left( \int_\Omega |\nabla \psi|^2 d\mu \right)^{1/2} 
\left( \int_\Omega |\nabla \varphi|^2 d\mu \right)^{1/2} ,
$$
which implies that
$$
\sup \left\{ \int_\Omega u\, \varphi \,d \mu :
\int_\Omega |\nabla \varphi|^2 d\mu = 1 \right\} \le 
\left( \int_\Omega |\nabla \psi|^2 d\mu \right)^{1/2} .
$$
On the other hand, from \eqref{uvarphi} we have
$$
\int_\Omega u\, \varphi\, d \mu = \left( \int_\Omega |\nabla \psi|^2 d\mu
\right)^{1/2}, \quad \text{when} \quad
\varphi = \psi \left( \int_\Omega |\nabla \psi|^2 d\mu \right)^{-1/2},
$$
so we conclude that $\|u\|_{\dot{H}^{-1}}^2 = \int_\Omega |\nabla \psi |^2
d\mu$ as was to be shown.
\end{proof}

\subsection{Divergence formulation}  \label{divf}
The $\dot{H}^{-1}(d\mu)$-norm can also be formulated as an optimization problem
over vector fields satisfying a divergence condition. We note that our
computational approach does not directly use this divergence formulation, and
that our purpose of stating the following result is for completeness and its
connection to other methods. 

\begin{proposition}
We have
$$
\| u\|_{\dot{H}^{-1}(d\mu)}^2 = \min_F \int_\Omega \| F \|^2 d\mu ,
$$
where the minimum is taken over vector fields $F$ with continuous first-order
partial derivatives that satisfy 
\begin{equation} \label{diveq}
\left\{
\begin{array}{cc}
-\frac{1}{f} \Div( f F) = u & \text{in } \Omega \\
n \cdot F = 0 & \text{on } \partial \Omega .
\end{array} \right. 
\end{equation}
\end{proposition}
Note that it will become clear from the proof that it would suffice to assume that
\eqref{diveq} holds in a weak sense.
\begin{proof} First, observe that if $\psi$ is a solution to the
elliptic boundary value problem \eqref{elliptic}, then $F = \nabla \psi$ is admissible to the minimization since
$$
-\frac{1}{f}\Div(f \nabla \psi) = -\Delta \psi - \nabla (\log f) \cdot \nabla
\psi  = L \psi =  u 
$$
and $n \cdot \nabla \psi = \partial_n \psi = 0$. And from Proposition
\ref{prop1} we have
$$
\| u\|_{\dot{H}^{-1}(d\mu)}^2 = \int_\Omega |\nabla \psi |^2 d\mu \ge \min_F
\int_\Omega \| F \|^2 d\mu .
$$
To complete the proof it suffices to show that
$$
\int_\Omega |\nabla \psi |^2 d\mu \le \min_F \int_\Omega \| F \|^2 d\mu .
$$
If we define $F = \nabla \psi + G$, then $G$ satisfies
$$
\left\{
\begin{array}{cc}
-\frac{1}{f} \Div( f G )= 0 & \text{in } \Omega \\
n \cdot G = 0 & \text{on } \partial \Omega .
\end{array} \right. 
$$
Expanding $F = \nabla \psi + G$ gives
$$
\int_\Omega \|F\|^2 d\mu = \int_\Omega (| \nabla \psi|^2 + 2 \nabla \psi
\cdot G + |G|^2 )\, d\mu.
$$
To complete the proof we will show that
$\int_\Omega \nabla \psi \cdot G \,d\mu = 0$. Observe that
$$
\int_\Omega \nabla \psi \cdot G \,d\mu =
- \int_\Omega \psi \frac{1}{f} \Div( f G) \,d\mu + 
\int_\Omega \frac{1}{f} \Div( \psi f G)\,d\mu .
$$
The first integral on the right hand side is zero since $\frac{1}{f} \Div (f
G) = 0$ in $\Omega$. The second integral on the right hand side is zero since
by the Gauss divergence theorem
$$
\int_\Omega \frac{1}{f} \Div( \psi f G) d\mu  =  \int_\Omega \Div(\psi f G)\,dx
=\int_{\partial \Omega} \psi f G \cdot n \,ds ,
$$
and $n \cdot G = 0$ by assumption. This completes the proof.
\end{proof}

\subsection{Witten Laplacian formulation} \label{wittenf}
The $\dot{H}^{-1}(d\mu)$-norm can also be defined in terms of a boundary value
problem involving the Witten Laplacian $H$, which is 
a Schr\"odinger operator of the form
\begin{equation} \label{Hsc}
H := -\Delta + V,
\end{equation}
where $V$ is a potential depending on $f$ defined by
$$
V := -\frac{1}{4} f^{-2} |\nabla f|^2 + \frac{1}{2} f^{-1} \Delta f  = f^{1/2}
\Delta f^{-1/2}.
$$
Alternatively, the Witten Laplacian $H$ can be defined by the similarity
transform 
\begin{equation} \label{Hsim}
H \psi =  f^{1/2} L (f^{-1/2} \psi),
\end{equation} 
which symmetrizes $L$ in the sense that the resulting operator $H$ is
self-adjoint with respect to $L^2(dx)$; for a discussion of the Witten Laplacian
and some spectral estimates see \cite{ColboisSoufiSavo2013}. In the following
Proposition, we give an elliptic equation involving $H$ that can be used to
compute the $\dot{H}^{-1}(d\mu)$-norm; this proposition is an immediate
consequence of the fact that the Schr\"odinger operator definition \eqref{Hsc}
is consistent with the similarity transform definition \eqref{Hsim}.

\begin{proposition} \label{propwitten}
Using the notation  $\tilde{u} = f^{1/2} u$ and $\tilde{\psi} = f^{1/2}
\psi$ we have
$$
\|u\|_{\dot{H}^{-1}(d\mu)} := \left(\int_\Omega \tilde{u}(x) \tilde{\psi}(x)
d x \right)^{1/2} ,
$$
where
$\tilde{\psi}$ is the solution to the elliptic boundary value problem:
$$
\left\{
\begin{array}{cc}
H \tilde{\psi} = \tilde{u} & \text{in } \Omega \\
\tilde{\psi} = 0 & \text{on } \partial \Omega .
\end{array} \right.  
$$
\end{proposition}
\begin{proof} This formulation is an immediate consequence of the identity 
$$
H\psi = f^{1/2} L (f^{-1/2} \psi) = 
-\Delta \psi + \left( -\frac{1}{4} f^{-2}
|\nabla f|^2 + \frac{1}{2} f^{-1} \Delta f \right) \psi,
$$
which is straightforward to verify: expanding $f^{1/2} L ( f^{-1/2} \psi)$ gives
$$
H \psi = -f^{1/2}( \text{div}( f^{-1/2} (\nabla \psi) - \frac{1}{2} f^{-3/2}
(\nabla f) \psi ) - \frac{\nabla f}{f^{1/2}} \cdot ( f^{-1/2} (\nabla \psi) -
\frac{1}{2} f^{-3/2} (\nabla f) \psi ),
$$
\begin{multline*}
 = \frac{1}{2} f^{-1} (\nabla f) \cdot (\nabla \psi) -
\Delta \psi + \frac{1}{2} f^{-1} (\nabla f) \cdot (\nabla \psi) -
\frac{3}{4} f^{-2} \|\nabla f\|^2 \psi + \frac{1}{2} f^{-1} (\Delta f) \psi  \\- f^{-1}(\nabla f) \cdot( \nabla \psi) +
\frac{1}{2} f^{-2}|\nabla f|^2 \psi ),
\end{multline*}
and after canceling terms we have
$$
H \psi =-\Delta \psi  -
\frac{1}{4} f^{-2} |\nabla f|^2 \psi + \frac{1}{2} f^{-1} (\Delta f) \psi .
$$
From the above calculation, it is also clear that
$$
V = f^{1/2} \Delta f^{-1/2}
$$
since all terms involving $\nabla f$ cancel. 
\end{proof}
\begin{remark}[Alternate form of the potential]
Using the fact that
$$
\Delta (-\log f) = -f^{-1} \Delta f + f^2 |\nabla f |^2,
$$
we can write $V$ as
$$
V = \frac{1}{4} | \nabla F |^2 - \frac{1}{2} \Delta F,
$$
where $F := -\log f$.
\end{remark}

\subsection{Linearization remainder estimate} \label{preciselin}
Recall that previously in \S \ref{linear} we gave the informal estimate
$$
W_2(\mu,\nu) =  \varepsilon \sqrt{\int_\Omega |\nabla \Psi |^2
d\mu} + \mathcal{O}(\varepsilon^2),
$$
where $\Psi$ satisfies 
$$
L \Psi = u, \quad \text{where} \quad
L := - \Delta - \nabla(\log f) \cdot \nabla.
$$
The purpose of this section, is to make this informal statement more precise; we
emphasize that the result proved in this section is for illustrative purposes:
results involving weaker assumptions and weaker regularity conditions are
possible.  

Let $\mu$ and $\nu$ be probability measures with densities $f$ and $g$ with
respect to the Lebesgue measure: $d\mu = f dx$ and $d\nu = g dx$.  Assume that
$$
\varphi(x) = \frac{|x|^2}{2} + \varepsilon \psi (x),
$$
where $\psi$ is a smooth function satisfying $\partial_n \psi = 0$ on $\partial
\Omega$. Further, assume that
$$
g(x) = (1 + \varepsilon u(x)) f(x),
$$
where $u$ is a smooth function. Assume that
$\varphi$ satisfies the nonlinear equation
\begin{equation} \label{nonlin}
f(x) = g(\nabla \varphi(x)) \det D^2 \varphi(x).
\end{equation}
Let $\Psi$ be the solution to the elliptic boundary value problem
$$
\left\{
\begin{array}{cc}
L \Psi = u & \text{in } \Omega \\
\partial_n \Psi = 0 & \text{on } \partial \Omega ,
\end{array} \right.  
\quad \text{where} \quad
L := - \Delta  - \nabla(\log f) \cdot \nabla.
$$
We have the following result.

\begin{proposition} \label{proplin}
Under the assumptions of \S \ref{preciselin} we have
$$
|W_2(\mu,\nu) - \varepsilon \|u\|_{\dot{H}^{-1}(d\mu)}| \le C_{\Omega,f,\psi,u} \varepsilon^2,
$$
where $C_{f,\psi,u}$ can be chosen in terms of almost everywhere upper bounds on:
$$
|\nabla f|,|H_f|,|\nabla \psi|,  |H_\psi|, |u|,|\nabla u|,|\Omega|,
\text{ and }  \lambda_1(L)^{-1},
$$
where $|\nabla f|$ denotes the magnitude of the gradient of $f$, $|H_f|$ denotes
the operator norm of the Hessian of $f$, $|\Omega|$ denotes the measure of
$\Omega$, and $\lambda_1(L)$ denotes the smallest positive eigenvalue of $L$.
\end{proposition}

We demonstrate this result numerically in \S \ref{basicex}.

\begin{proof}[Proof of Proposition \ref{proplin}]
By the Lagrange remainder formulation of Taylor's Theorem, and \eqref{nonlin} we have
\begin{equation} \label{eqff}
f = (1 + \varepsilon u + \varepsilon^2 R_1)) (f +
\varepsilon \nabla f \cdot \nabla \psi + \varepsilon^2 R_2) 
    (1 + \varepsilon \Delta \psi + \varepsilon^2 \det H_\psi),
\end{equation}
where the remainder functions $R_1,R_2$ can be expressed by
$$
R_1(x) = \nabla u(\xi_1) \cdot \nabla \psi(x), \quad   \text{and} \quad
R_2(x) = (\nabla \psi(x))^\top H_f(\xi_2) \nabla \psi(x),
$$
where $\xi_1,\xi_2$ are points on the line segment
between $x$ and $x + \varepsilon \nabla \psi(x)$.
It follows that
\begin{equation} \label{eqrrs}
\varepsilon (-\Delta \psi - \nabla (\log f) \cdot \nabla \psi - u) f = \varepsilon^2 R,
\end{equation}
where the remainder function $\varepsilon^2 R$ consists of all terms in the
expansion of the right hand side of \eqref{eqff} that include $\varepsilon$ to
power at least $2$. By the definition of $L$ and $\Psi$ we can rewrite
\eqref{eqrrs} as
$$
 (L \psi - L \Psi) f = \varepsilon R.
$$
Multiplying both sides of this equation by $\psi + \Psi$ and integrating over $\Omega$ gives
$$
\int_\Omega (\psi + \Psi)(L \psi - L \Psi) d\mu = \varepsilon \int_\Omega (\psi + \Psi) R \,dx.
$$
Using \eqref{Ldiv} to rewrite the left hand side gives
$$
\int_\Omega |\nabla \psi|^2 d\mu - \int_\Omega |\nabla \Psi|^2  d\mu =
\varepsilon \int_\Omega (\psi + \Psi) R \,dx.
$$
By \eqref{eq3} and \eqref{elliptich1} it follows that
\begin{equation} \label{eqlast}
W_2(\mu,\nu)^2 - \varepsilon^2 \|u\|_{\dot{H}^{-1}(d\mu)}^2 = \varepsilon^3
\left( \int_\Omega \psi R dx + \int_\Omega \Psi R dx  \right).
\end{equation}
We can bound the $L^2$ norm of $\Psi$ by the $L^2$ norm of $u$ and the
inverse of the smallest positive eigenvalue of $L$; therefore, we can complete
the proof by using Cauchy-Schwarz and almost everywhere bounds on all other
quantities.
\end{proof}
\begin{remark}
 We note that it is possible to obtain various estimates on $|W_2(\mu,\nu) -
\varepsilon \|u\|_{\dot{H}^{-1}(d\mu)}|$ from \eqref{eqlast} in
terms of $L^p$ norms of the quantities 
$\nabla f,H_f,\nabla \varphi,  H_\psi, u,$ and $\nabla u$ instead of almost everywhere bounds.
Moreover, results that guarantee bounds on the solution of general Monge-Amp\'ere
equations could be  used to provide bounds for $\nabla \psi$ and $H_\psi$ in
terms of $f$ and $u$, for example see \cite{Lions1986}.
\end{remark}

\subsection{Fourier transform and weighted Sobolev norms} \label{fourier}
In this section, we discuss a family of weighted Sobolev norms defined by 
Engquist, Ren, and Yang  \cite{EngquistRenYang2020} using the Fourier transform
$$
\hat{f}(\xi) = \int_{\mathbb{R}^N} f(x) e^{-2\pi i x \cdot \xi} dx.
$$
In \cite{EngquistRenYang2020} the authors compare the $W_2$ metric to a family
of weighted Sobolev norms defined in Fourier domain for the purpose of inverse
data matching; in particular, in [Eq. 8 and Remark 2.1 of
\cite{EngquistRenYang2020}] they define
\begin{equation} \label{abs}
\| u \|_{\bar{H}^{s}(w)}^2 = \int_{\mathbb{R}^N} 
\big| \widehat{w} * \widehat{u}_s \big|^2 d\xi \quad \text{where} \quad  
\widehat{u}_s(\xi) := (2\pi|\xi|)^s \widehat{u}(\xi),
\end{equation}
where $\widehat{w}$ and $\widehat{u}$ denote the Fourier transform of $w$ and
$u$, respectively, and  $*$ denotes convolution
$$
(f * g)(x) = \int_{\mathbb{R}^N} f(x-y) g(y) dy.
$$
Here we refer to the norm defined in \eqref{abs} as the $\bar{H}^s(w)$-norm
to avoid confusion with the $\dot{H}^1(d\mu)$-norm, which we defined in
\eqref{eqh1d} by
$$
\|\varphi\|_{\dot{H}^1(d\mu)}^2 := \int_\Omega |\nabla \varphi |^2 d\mu,
$$
and the $\dot{H}^{-1}(d\mu)$-norm defined in \eqref{Hdefn} by
$$
\|u\|_{\dot{H}^{-1}(d\mu)}  := \sup \left\{ \int_\Omega u\, \varphi\, d\mu :
\|\varphi\|_{\dot{H}^1(d\mu)} = 1 \right\},
$$
also see [\S 7.6 of \cite{Villani2003}].
First, observe that if $s=1$ and $w=1$ is the constant function, then
$\widehat{w} = \delta$ is the Dirac delta distribution and
$$
\| u \|_{\bar{H}^{s}(w)}^2 = \int_{\mathbb{R}^N} 
\big| (2\pi |\xi|) \widehat{u}(\xi) \big|^2 d\xi = 
 \int_{\mathbb{R}^N} \big| (2\pi i \xi) \widehat{u}(\xi) \big|^2 d\xi =
\int_{\mathbb{R}^N} |\nabla u|^2 d x,
$$
where the final inequality follows from the Plancherel theorem.
However, if $s =1$ and $w$ is arbitrary, then in general
$$
\| u \|_{\bar{H}^{s}(w)}^2 = \int_{\mathbb{R}^N} |\widehat{w} * \widehat{u}_1|^2
d\xi = \int_{\mathbb{R}^N} |u_1 |^2 w^2 dx \not = 
 \int_{\mathbb{R}^N} |\nabla u |^2 w^2 dx = \|u\|_{\dot{H}^1(d\mu)}^2,
$$
where $u_1$ denotes the inverse Fourier transform of $\widehat{u}_1(\xi) = 2\pi
|\xi| \widehat{u}(\xi)$, and $\mu$ is assumed to have density $w^2$ with respect
to Lebesgue measure. It does follow from the Plancherel theorem that 
$$
\int_{\mathbb{R}^N} |u_1 |^2 dx = 
\int_{\mathbb{R}^N} |\nabla u |^2 dx.
$$
However, in general, the functions $|u_1|^2$ and $|\nabla u|^2$ are not
equal, and thus in general their integrals against $w^2$ are not equal;
in particular, their integrals against $w^2$ can be very different when $w^2$ is
localized in space.  Roughly speaking, the issue is that taking the absolute
value of $\xi$ does not commute with taking the convolution.  It is possible
to define the
$\dot{H}^1(d\mu)$-norm in Fourier domain by defining $\widehat{u}(\xi) := 2\pi i
\xi \widehat{u}$; however, this does not seem to lead to a viable way
to compute the dual $\dot{H}^{-1}(d\mu)$-norm except in dimension $N=1$; we note
that quadratic Wasserstein distance also has a simple characterization in
$1$-dimension, see Remark \ref{w21d}.

\begin{remark}[$W_2$ in $1$-dimension] \label{w21d}
Let $\mu$ and $\nu$ be measures on $\mathbb{R}$ with densities $f$ and
$g$ with respect to Lebesgue measure: $d\mu = f dx$ and $d\nu = g\,
dx$.  Let $F, G : \mathbb{R} \rightarrow [0,1]$ denote the cumulative
distribution functions:
$$
F(x) = \int_{-\infty}^x f(y) dy, \quad \text{and} \quad G(x) = \int_{-\infty}^x
g(y) dy.
$$
If $F^{-1}$ and $G^{-1}$ are the pseudo-inverse of $F$ and $G$ defined by
$$
F^{-1}(t) = \min \{ x \in \mathbb{R} : F(x) \ge t\}
\quad \text{and} \quad
G^{-1}(t) = \min \{ x \in \mathbb{R} : G(x) \ge t\},
$$
then
$$
W_2(\mu,\nu)^2  = \int_0^1 (F^{-1}(t) - G^{-1}(t))^2 dt,
$$
see for example [Remark 2.30 of \cite{PeyreCuturi2019}].
\end{remark}


\section{Computation and regularization} \label{compsec}
 In this section, we consider the connection between the
$\dot{H}^{-1}(d\mu)$-norm and the Witten Laplacian, see \S \ref{wittenf},
from a computational point of view. Recall that the Witten Laplacian is a
Schr\"odinger  operator of the form $$
H = -\Delta + V ,
$$
where $V$ is a potential. Roughly speaking, the advantage of considering this
formulation is that all the complexity of the problem has been distilled into
the potential $V$, which can be regularized to manage the
computational cost.  This section is organized as follows. First, in \S
\ref{spectral} we consider a spectral decomposition of $-\Delta$ by its
Neumann eigenfunctions and define the fractional Laplacian $(-\Delta)^\gamma$.
Second, in \S \ref{preconditioning}, we change variables using fractional
Laplacians to precondition our elliptic equation involving $H$.  Third, in \S
\ref{smoothout}, we observe how using the heat equation to define a smoothed
version of $V$ can control the condition number of our problem.  Finally, in \S
\ref{compcost} we discuss the computational cost of the described method.

\subsection{Spectral decomposition of the Laplacian} \label{spectral}
Suppose that $\Omega$ is a bounded convex domain.  Recall that $\lambda$ is a
Neumann eigenvalue of the Laplacian $-\Delta$ on $\Omega$ if there is a
corresponding
eigenfunction $\varphi$ such that
$$
\left\{ \begin{array}{cc}
-\Delta \varphi = \lambda \varphi  & \text{in } \Omega \\
\partial_n \varphi = 0 & \text{on } \partial \Omega,
\end{array} \right.
$$
where $\partial \Omega$ denotes the boundary of $\Omega$, and
$n$ denotes an exterior unit normal to the boundary. The Neumann eigenvalues of
the Laplacian are nonnegative real numbers that satisfy 
$$
0 = \lambda_0 < \lambda_1 \le \lambda_2 \le \cdots \le \lambda _k \le \cdots
\nearrow +\infty,
$$
and the corresponding eigenfunctions $\{\varphi_k\}_{k=0}^\infty$ form an
orthogonal basis of square integrable functions on $\Omega$.
We can use this basis of Neumann eigenfunctions to define
the fractional Laplacian $(-\Delta)^\gamma$ for $\gamma \in \mathbb{R}$ and
$\psi = \sum_k \alpha_k \varphi_k$ by
$$
(-\Delta)^\gamma \psi = \alpha_0 \varphi_0 + \sum_{k > 0} \lambda_k^\gamma \alpha_k
\varphi_k,
$$
where we include the constant term $\alpha_0 \varphi_0$, independent of
$\gamma$, so that the operator is well-defined and invertible for both positive
and negative $\gamma$. With this definition, the operator $(-\Delta)^\gamma$ is
invertible, which will become relevant in the following section; in particular,
we will use the invertible operator $(-\Delta)^\gamma$ to precondition the
elliptic equation $H \psi = u$.

\subsection{Preconditioning the elliptic equation} \label{preconditioning} 
Recall that our goal is to solve the elliptic equation
\begin{equation} \label{Hpsiuu}
H \psi = (-\Delta + V) \psi = u,
\end{equation}
with Neumann boundary conditions; note that to simplify notation we dispense
with the tilde notation $\tilde{\psi}$ and $\tilde{u}$ from \S \ref{wittenf} and
just
write $\psi$ and $u$. In order to precondition \eqref{Hpsiuu} we define
$U$ and $\Psi$ by
$$
U := (-\Delta)^{1/2} u \quad \text{and} \quad
\Psi := (-\Delta)^{1/2} \psi.
$$ 
It follows that
\begin{equation}\label{20}
A \Psi = U
\end{equation}
where
$$
A  := \Id - P_1 + (-\Delta)^{-1/2} V (-\Delta)^{-1/2},
$$
where $P_1$ denotes the projection onto the space of constant functions,
$$
(P_1 \Psi)(x) = \frac{1}{|\Omega|} \int_\Omega \Psi(y) dy,
$$
and $\Id$ denotes the identity operator.
We remark that the projection $P_1$ is necessary in the definition of $A$ since
we have defined $(-\Delta)^{-1/2}$ to preserve constant functions, while
the Laplacian $-\Delta$ destroys constant functions. 

Since $(-\Delta)^{-1/2}$ is invertible the dimension of the null space of $A$ is
the same as the dimension of the null space of $H$, which is $1$-dimensional. In
particular, we have 
$$
H f^{-1/2} = 0, \quad \text{and} \quad A (-\Delta)^{1/2} f^{-1/2} = 0.
$$
Let $\lambda_1(H)$ and $\lambda_1(A)$ denote the smallest positive eigenvalue of
$H$ and $A$, respectively. If $\psi_1(H)$ is a normalized eigenvector associated
with $\lambda_1(H)$, then it follows from the Courant-Fisher Theorem that
$$
\lambda_1(A) \ge c \lambda_1(H),
$$
where $c = (\int_\Omega \|\nabla \psi_1(H)\|^2 dx)^{-1}$. In the following, we
treat $c$ as a fixed constant, which empirically we find is the case.
Under this assumption, the condition number of $A$ on
the space of functions orthogonal to $(-\Delta)^{1/2} f^{-1/2}$ is bounded by
the operator norm of $A$, which satisfies
$$
\|A\|^2 \le 1 + \|V\|_{L^\infty}, 
$$
If $f$ is an arbitrary smooth positive function, then $V$ could still take very
large values, which could make our problem ill-conditioned. In the following
section, we introduce a definition of $V$ that includes smoothing which can be
used to control its $L^\infty$-norm.

\subsection{Smoothing when defining the potential} \label{smoothout}
Recall that the potential $V$ in the definition of $H$ can be defined by
$$
V := f^{-1/2} \Delta f^{1/2}.
$$
The basic idea is to run the heat equation on $f^{1/2}$ and use the resulting
smoothed function to define the potential $V$.  Given a function $f$, we 
define a $1$-parameter family of norms parameterized by $\tau > 0$ as follows.
Let $\{\lambda_k\}_{k=0}^\infty$ and $\{\varphi_k\}_{k=0}^\infty$ denote the
Neumann Laplacian eigenvalues and eigenfunctions, see \S \ref{spectral}.  We
define the Neumann heat kernel $e^{-\tau \Delta}$ for a function $\psi = \sum_k
\alpha_k \varphi_k$ by 
$$
e^{-\Delta t} \psi = \sum_k e^{-\tau \lambda_k} \alpha_k \varphi_k.
$$
Next, we define the smoothed potential  $V_\tau$ by
\begin{equation} \label{smoothV}
 V = f^{-1/2}_\tau \Delta f^{1/2}_\tau, \quad \text{where} \quad f_\tau :=
(e^{-\tau \Delta} f^{1/2})^2,
\end{equation}
and define the corresponding operator $H_\tau$ by
$$
H_\tau := -\Delta + V_\tau.
$$
By \S \ref{wittenf}, the operator $H_\tau$ defines a
$\dot{H}^{-1}(d\mu_\tau)$-norm, where $\mu_\tau$ is the measure with density
$f_\tau$. Observe that if $\tau = 0$ then
$\|u\|_{\dot{H}^{-1}(d\mu_\tau)} = \|u\|_{\dot{H}^{-1}(d\mu)}$, while when
$\tau \rightarrow \infty$ then $\|u\|_{\dot{H}^{-1}(d\mu_\tau)} \rightarrow
\|u\|_{\dot{H}^{-1}(dx)}$. In particular, we have
$$
\|V_\tau\|_{L^\infty} \rightarrow 0, \quad \text{as} \quad \tau \rightarrow
\infty,
$$
so the parameter $\tau$ can be used to control the condition number of $H$, and
hence can be used to control the computational cost as is discussed in the
following section.

\subsection{Computational cost} \label{compcost}
Computing the $\dot{H}^{-1}(d\mu_\tau)$-norm using the operator
$$
H_\tau = -\Delta + V,
$$
involves solving an elliptic equation involving $H_\tau$. By \S
\ref{preconditioning} this equation can be preconditioned by a change of
variables resulting in a linear system 
$$
A_\tau \Psi = U,
$$
where $A_\tau$ is an operator with condition number
$\mathcal{O}(\sqrt{1+\|V_\tau\|_{L^\infty}})$.  Since $A_\tau$ is positive
definite on the space orthogonal to its null space, and since $U$ is contained
in this space, we can use Conjugate Gradient to solve this
linear system to a fixed precision $\varepsilon >0$ with computational cost 
$$
C_\text{solve} = \mathcal{O}\left( C_A  \sqrt{1 +
\|V_\tau\|_{L^\infty}} \right),
$$
where $C_A$ is the cost to apply $A$. The operator $A$ can be applied quickly if we can efficiently change between the standard basis
and the basis of Neumann Laplacian eigenfunctions. In the following remark, we
discuss the case $\Omega = [0,1]^2$, where this transformation can be performed
by a Discrete Cosine Transform (DCT).

\begin{remark}[Spectral decomposition of Laplacian on unit square]
In the case $\Omega = [0,1]^2$, the Neumann Laplacian eigenvalues and
eigenfunctions can be  indexed by $k = (k_1,k_2) \in \mathbb{Z}_{>0}^2$ and are
of the form
$$
\lambda_k = k_1^2 + k_2^2
\quad \text{and} \quad
\varphi_k(x) = c_{k_1} c_{k_2} \cos(\pi k_1 x_1) \cos(\pi k_2 x_2),
$$
where $x = (x_1,x_2)$ and $c_{k_1}$ and $c_{k_2}$ are constants to normalize
$\varphi_k$ to have unit $L^2$ norm: $c_{k_1} = 1/\sqrt{2}$ if $k_1 > 0$
and $c_{k_1} = 1$ if $k_1 = 0$.  Thus, expanding a function on the unit square
in these Neumann eigenfunctions is equivalent to expanding a function in the
double cosine series, which can be efficiently achieved by the Discrete Cosine
Transform (DCT). In particular, if a function on the unit square $[0,1]^2$ is
represented by an $n \times n$ array, then  the computational cost of expanding
in a double cosine series using the DCT is $\mathcal{O}(n^2 \log n)$ operations.
\end{remark}

\section{Numerical examples} \label{numericssec}

In this section we describe a numerical algorithm for using the Witten
Laplacian to compute a local linear approximation of $W_2$ distance via the
$\dot{H}^{-1}(d\mu_\tau)$-norm. We use the analytical tools of the previous
sections, and demonstrate the method on several numerical examples. In
particular, this section is organized as follows: First, in \S \ref{implement} we
describe the implementation of the algorithm and provide a link to code. 
In \S \ref{propwass} we include analytical results about Wasserstein distance
for Gaussian distributions and translations
that we will use to interpret the numerical results. Third, in \S \ref{basicex},
we provide an initial numerical example for Gaussian distributions that
illustrates the result of Proposition \ref{proplin}.
Next, in \S \ref{translateex} we provide illustrations of how the 
linearization approximates
Wasserstein distance for translations. Fifth, in \S \ref{varianceex} we
include visualizations of how the linearization approximates Wasserstein distance for 
changes in variance. Finally, in \S \ref{secembed} we present an example of computing 
an embedding of the $\dot{H}^{-1}(d\mu)$-norm into $L^2$.

\subsection{Implementation} \label{implement}

\begin{algorithm}[Linearized $W_2$ via Witten Laplacian] \label{alg1}
We first compute the Witten potential, $V$, and 
then solve the resulting partial differential
equation by converting it to a symmetric linear 
system which is solved using conjugate gradient. 
\begin{enumerate}[\quad 1)]
    \item Compute the potential $V_\tau$ using the smoothing procedure of
section \ref{smoothout}.
    \item Solve the linear system 
    $$ A \Psi = U $$ 
    using conjugate gradient where $A$ is defined by (\ref{20}).
    \begin{enumerate}[\quad a)]
    \item 
        The discretized operator $A$ of (\ref{20}) can be applied to a function, 
        $f$,
        tabulated on an equispaced grid by first approximating $f$
        as a $2$-dimensional cosine expansion of the form 
        $$ f(x_1, x_2) \approx \sum_{k_1, k_2 = 0}^{n-1} \alpha_{k_1, k_2} \cos(\pi k_1 x_1)\cos(\pi k_2 x_2)$$
        where $n$ is the number of function tabulations in each spatial dimension. 
        The coefficients $\alpha_{k_1, k_2}$
        are computed with a Discrete Cosine Transform (DCT), which 
        requires $O(n^2 \log{n})$ operations. 
    \item The operator $\Delta^{-1/2}$ of $A$ is 
        applied to a cosine expansion via pointwise multiplication of 
        the coefficients. For example, $\Delta^{-1/2} \cos(m x_1) = -\frac{1}{m} \cos(mx_1)$. 
    \item Pointwise multiplication by $V_\tau$  in spatial domain 
        is then performed with an inverse DCT, followed 
        by pointwise multiplication in the spatial domain. 
    \item Conjugate gradient is iterated until convergence up to some desired 
        error tolerance.
    \end{enumerate}
\end{enumerate}
\end{algorithm}

We implemented the preceding algorithm in Python, and have provided
publicly available codes with the implementation accessible at
\url{https://github.com/nmarshallf/witten_lw2}.

\subsection{Analytic formulas for \texorpdfstring{$W_2$}{W2} for Gaussian
distributions and translations} \label{propwass}

Let $\mu$ and $\nu$ be measures on $\mathbb{R}^N$ with densities $f$
and $g$ with respect to Lebesgue measure: $d\mu = f dx$ and $d\nu =
g dx$. Assume that $f$ is a Gaussian function with mean $m_f$ and
diagonal covariance $\Sigma_f = \diag(\sigma_f^2)$, where $\sigma_f^2
=(\sigma_{f,1}^2,\ldots,\sigma_{f,N}^2)$
$$
f(x) = \frac{1}{(2\pi)^{d/2} (\det \Sigma_f)^{1/2}} \exp \left( -\frac{1}{2}
(x-m_1)^\top \Sigma_f^{-1} (x-m_f)    \right).
$$
Similarly, assume that $g$ is a Gaussian function with mean $m_g$ and covariance
$\Sigma_g = \diag(\sigma_g^2)$. Then,
\begin{equation} \label{w2gauss}
W_2(\mu,\nu)^2 = |m_f - m_g|_2^2 + | \sigma_f - \sigma_g |_2^2.
\end{equation}
That is, the square of the quadradic Wasserstein distance between Gaussian
distributions with diagonal covariance matrices is equal to the square of the
distance between their means plus the square of the distance between their
standard deviations, see [Remark 2.31 of \cite{PeyreCuturi2019}] for a more
general result.

The dependence of quadratic Wasserstein distance on the distance between
means for Gaussian distributions is a special case of a general translation
property.  Let $\mu$ and $\nu$ be two measures on $\mathbb{R}^N$ that
have the same mean
$$
\int_{\mathbb{R}^N} x d\nu = \int_{\mathbb{R}^N} x d\mu.
$$
Let $T^v: \mathbb{R}^N \rightarrow \mathbb{R}^N$ denote the translation
operator $T^v: x \mapsto x + v$. Suppose that $\nu_v$ denote a translation of
$\nu$ by $v$; more formally, $\nu_v := T^v_\# \nu$ where $\#$ denotes the push
forward.  Then quadratic Wasserstein distance satisfies the following relation:
\begin{equation}  \label{w2trans}
W_2(\mu,\nu_v)^2 = W_2(\mu,\nu)^2 + |v|_2^2,
\end{equation}
That is, if two measures have the same mean and one measure is translated
distance $|v|$, then the square of the quadratic Wasserstein distance between
the measures  increases by $|v|^2$, see [Remark 2.19 of \cite{PeyreCuturi2019}]
for a slightly more general statement of this
translation result.

\subsection{Numerical example: linearization of \texorpdfstring{$W_2$}{W2} for
Gaussian distributions} \label{basicex}

In this section, we demonstrate that our code satisfies the result of
Proposition \ref{proplin} using Gaussian distributions and \eqref{w2gauss}.
Let $\mu$ and $\nu$ be measures on $\mathbb{R}^N$ with densities $f$
and $g$ with respect to Lebesgue measure: $d\mu = f dx$ and $d\nu =
g dx$. We define $f : [0,1]^2 \rightarrow \mathbb{R}$ by
\begin{equation} \label{gaussss}
f(x) = \frac{1}{2 \pi (\det \Sigma_f)^{1/2}} \exp\left( -\frac{1}{2}(x -
\mu_f)^\top \Sigma_f^{-1} (x - \mu_f) \right),
\end{equation}
with 
$$
\mu_f = (1/2,\, \, 1/2)^\top
\qquad \text{and} \qquad
\Sigma_f = 
\begin{pmatrix}
    1 / 16 & 0 \\
    0 & 1 / 14
\end{pmatrix},
$$
and $g : [0,1]^2 \rightarrow \mathbb{R}$  by
$$
g(x) = \frac{1}{2 \pi (\det \Sigma_g)^{1/2}} \exp\left( -\frac{1}{2}(x -
\mu_g)^\top \Sigma_g^{-1} (x - \mu_g) \right), 
$$
with
$$
\mu_g = \mu_f + (0.001, \,\, 0.002)^\top
\qquad \text{and} \qquad
\Sigma_g = \Sigma_f + 
\begin{pmatrix}
    0.001 & 0 \\
    0 & 0.003
\end{pmatrix}.
$$
The covariances $\Sigma_f$ and $\Sigma_g$ are chosen such that, for numerical
purposes up to precision $10^{-16}$, the functions $f$ and $g$, are essentially
supported on $[0,1]^2$ and thus both are probability densities that integrate to
$1$.  For this numerical example, we use the above Python implementation of
Algorithm \ref{alg1} for the functions $f$ and $g$ tabulated on a $513 \times
513$ equispaced grid on $[0, 1]^2$, we define $u = (f - g)/f$,
see \S \ref{smoothout}.  Recall that  Proposition \ref{proplin} says that if
$\|u\|_{\dot{H}^{-1}(d\mu)} = \varepsilon$, then 
$$
|W_2(\mu,\nu) -
\|u\|_{\dot{H}^{-1}(d\mu)} | = \mathcal{O}(\varepsilon^2).  
$$
The implementation gives
\begin{equation} \label{verifyprop1}
\|u\|_{\dot{H}^{-1}(d\mu)} \approx 1.2397 \times 10^{-3};
\end{equation}
using \eqref{w2gauss} we find that
\begin{equation} \label{verifyprop2}
|W_2(\mu,\nu) -  \|u\|_{\dot{H}^{-1}(d\mu)} | \approx 6.8598 \times 10^{-6}.
\end{equation}
Thus, \eqref{verifyprop1} and \eqref{verifyprop2} provide a numerical
demonstration of Proposition \eqref{proplin}.

\subsection{Numerical example: visualizing the linearization for translations}
\label{translateex} 

In this section, we visualize how different metrics compare to $W_2$ by
considering a subset of the sphere $\{ \nu : W_2(\mu,\nu)= \varepsilon\}$; in
particular, we consider the subset of this sphere that consists of translated
versions of $\mu$. Fix $\varepsilon > 0$, let $S$ denote the
unit circle $S := \{ v \in \mathbb{R}^2 : \|v\|_2 = 1 \}$, and observe that
\begin{equation} \label{w2slice}
\left\{ W_2(\mu,\mu_{\varepsilon v}) v \in \mathbb{R}^2 : v \in
S \right\} = \varepsilon S,
\end{equation}
where $\mu_{\varepsilon v}$ is the translation of $\mu$ by $\varepsilon v$; the
fact that this set is equal to $\varepsilon S$ follows from \eqref{w2trans}.
 In the following, we define analogs of the set defined in the left hand side of
\eqref{w2slice}, where the $W_2$ metric is replaced by our linearization, the
unweighted Sobolev norm, and the Euclidean norm, respectively. By plotting these
sets, we can understand how these metrics distort slices of small spheres with
respect to the $W_2$ metric. Let $\mu$ be a measure with density $f$ with
respect to the Lebesgue measure: $d \mu = f dx$. Suppose that $\mu_{\varepsilon
v}$ is the translation of $\mu$ by $\varepsilon v$, which is the measure with
density $f_{\varepsilon v}(x) := f(x + \varepsilon v)$.
First, we use the weighted negative homogeneous Sobolev norm based on the
regularized Witten Laplacian formulation described  in \S \ref{compsec} to
define
$$
T_\text{witten} := \left\{ \|(f-f_{\varepsilon
v})/f_\tau\|_{\dot{H}^{-1}(d\mu_\tau)} v \in \mathbb{R}^2 : v \in S \right\},
$$
second, we use the unweighted Sobolev norm to define
$$
T_\text{sobolev} := \left\{ \|f - f_{\varepsilon v}\|_{\dot{H}^{-1}(dx)} v \in \mathbb{R}^2
: v \in S \right\},
$$
and third, we use the Euclidean norm to define
$$
T_\text{euclid} := \left\{ \|f - f_{\varepsilon v}\|_{L^2(dx)} v \in \mathbb{R}^2 : v \in S
\right\}.  
$$
For this numerical example, we use the function $f$ plotted in Figure
\ref{fig01}, see \S \ref{intro}. This function $f : [0,1]^2
\rightarrow \mathbb{R}$ is defined by
$$
f(x) = \frac{1}{c} \exp(9 x_1) (\cos(16 \pi x)+1) \zeta(x),
$$
where $\zeta$ is a bump function supported in $[.1,.9]^2$ that is equal to $1$
on $[.2,.8]^2$, and $c$ is a constant that normalizes $f$ so that it is a
probability density; given $f$ we define the potential $V_\tau$, see Figure
\ref{fig01}.  We plot the sets $T_\text{witten}, T_\text{sobolev},$ and
$T_\text{euclid}$ in Figure \ref{fig02}. Note that the set $\varepsilon S$ is
included for reference and is plotted using a dotted line in  the plots of
Figure \ref{fig02}.

\begin{figure}[ht!]
\centering
\begin{tabular}{ccc}
\includegraphics[width =.3\textwidth]{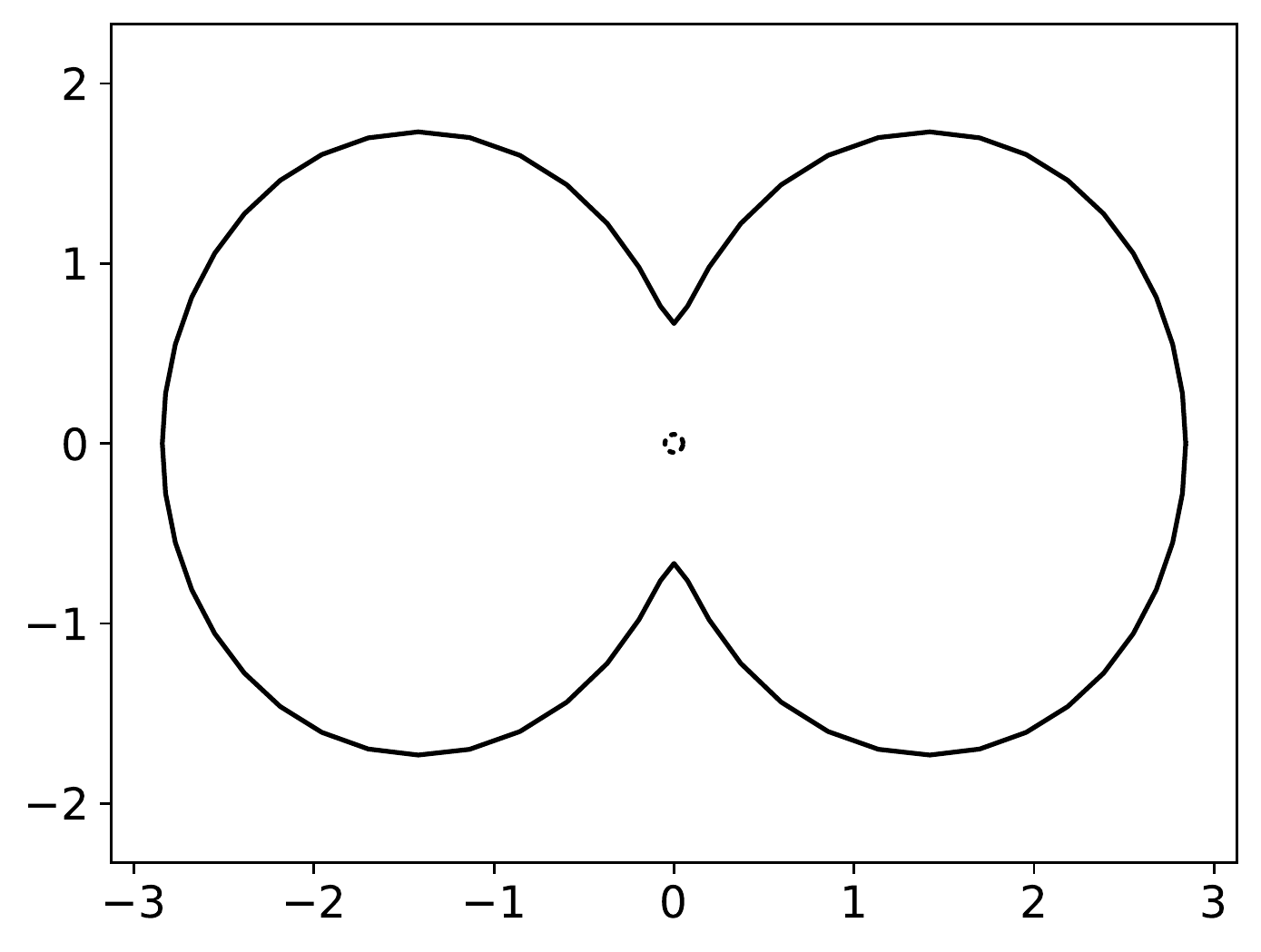} &
\includegraphics[width =.3\textwidth]{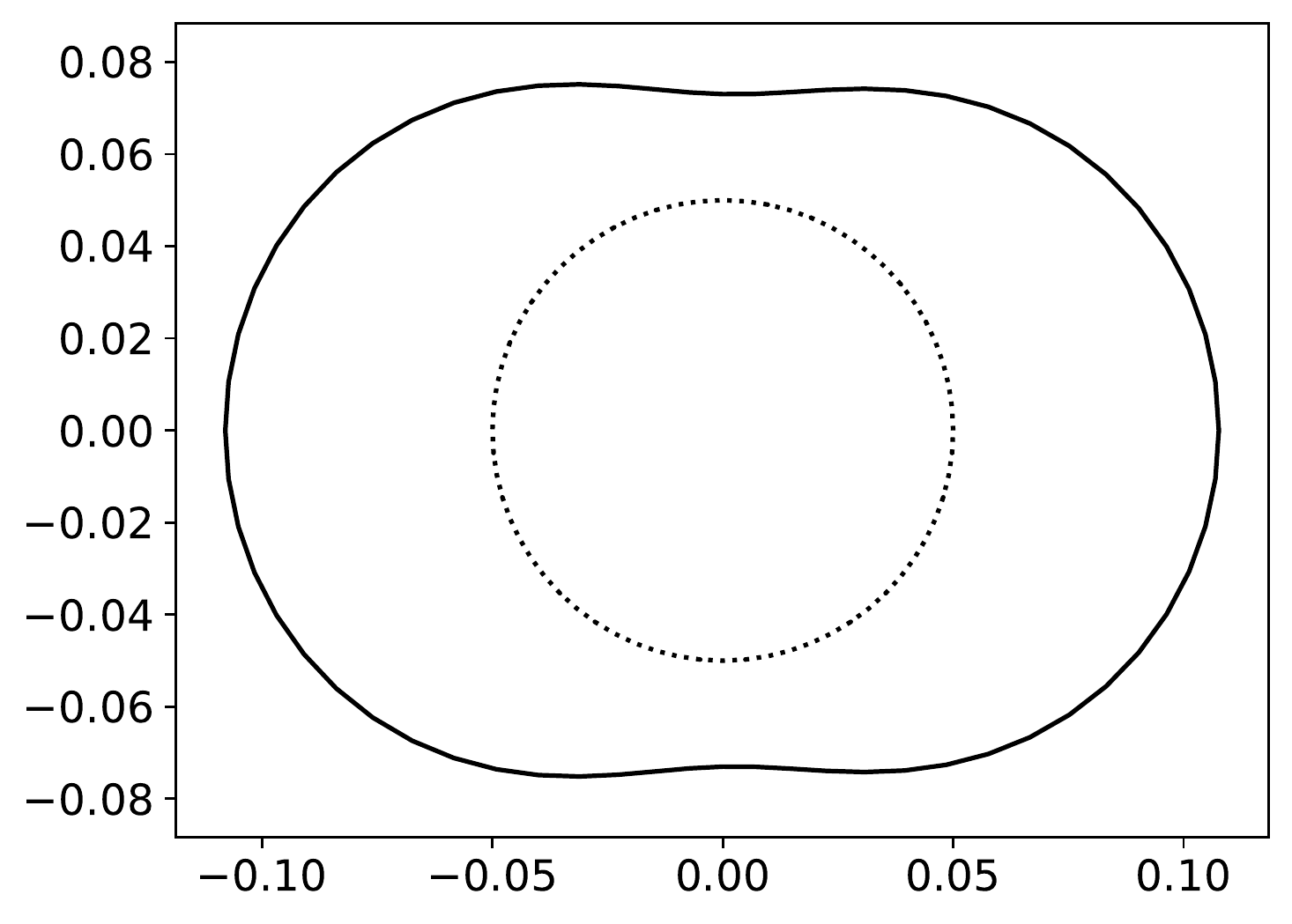} &
\includegraphics[width =.3\textwidth]{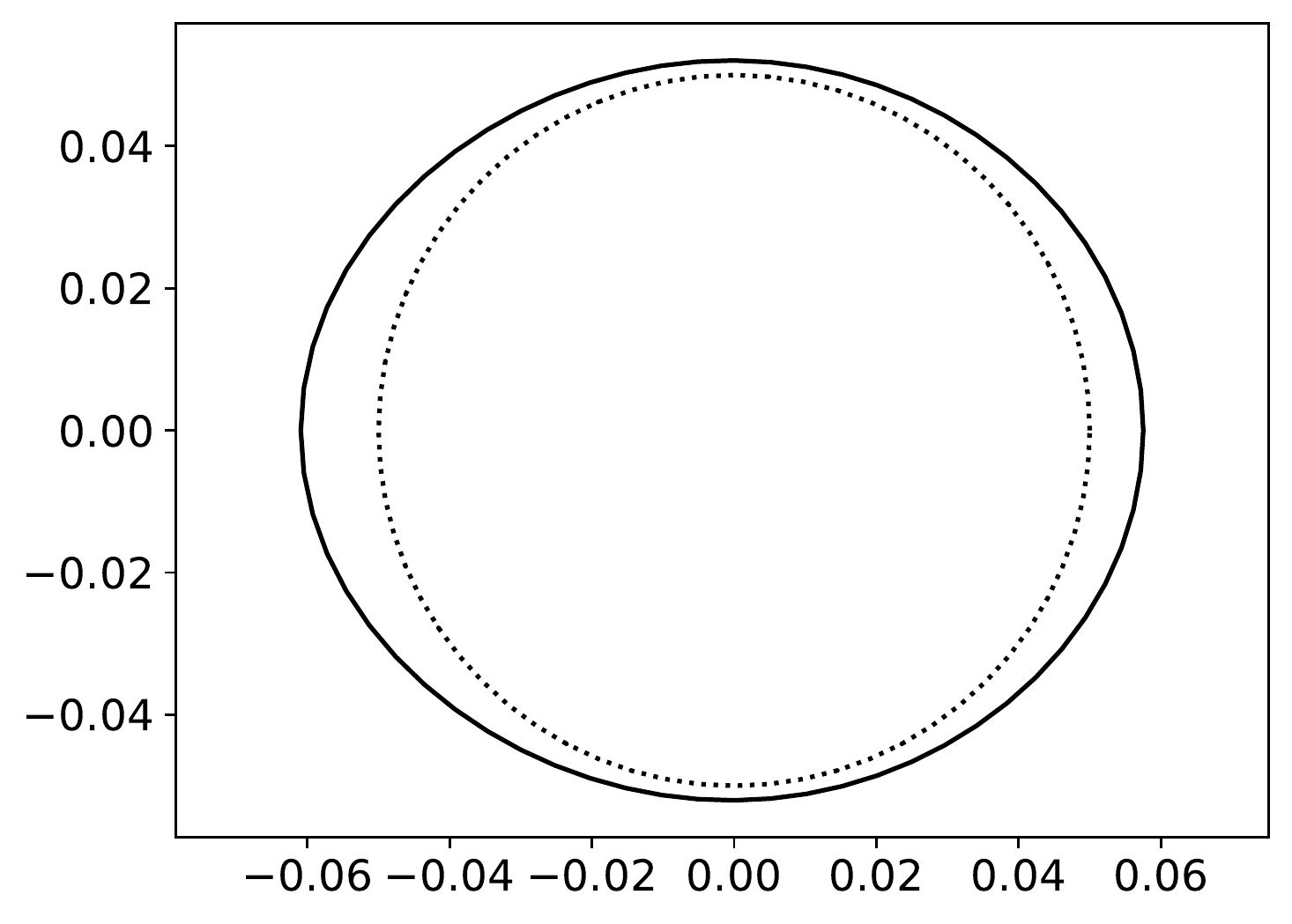}
\end{tabular}
\caption{$T_\text{euclid}$ (left), $T_\text{sobolev}$ (middle), and
$T_\text{witten}$ (right).}
\label{fig02}
\end{figure}

First, consider the plot of $T_\text{euclid}$ in Figure \ref{fig02}. Since the
probability measures $f$ and $f_{\varepsilon v}$ are probability densities, they
can be thought of as being normalized to have $L^1$-norm equal to $1$, which is
the reason that the scale of $T_\text{euclid}$ is much larger than $\varepsilon
S$ which appears as a dot. The shape of $T_\text{euclid}$ can be interpreted as
follows: if the image $f$ is translated up, then the vertical stripes will
mostly overlap, see Figure \ref{fig01}, resulting in a small change in the
Euclidean distance.  In contrast, if the image is shifted left, then the strips
will become misaligned resulting in a large change in the Euclidean distance;
this explains the barbell shape of the set $T_\text{euclid}$. Next, consider the
plot of $T_\text{sobolev}$ in Figure \ref{fig02} corresponding to the unweighted
Sobolev norm, which partially corrects the scaling. Finally, the plot of
$T_\text{wittin}$ which is the linear approximation of Wasserstein distance
computed using the method described in this paper nearly recovers the circle
with only a small deformation.

\subsection{Numerical example: visualizing effect of changing variance}
\label{varianceex}

In this section, we again visualize how different metrics compare to $W_2$ by
considering a subset of the sphere $\{ \nu : W_2(\mu,\nu)=
\varepsilon\}$. By assuming that the density $f$ of $\mu$ is a Gaussian function
with a diagonal covariance matrix we can consider the subset of $\{ \nu :
W_2(\mu,\nu)= \varepsilon\}$ consisting of Gaussian distributions with the same
mean, but whose diagonal covariance matrix is different. 
Let $f$ be a Gaussian function  centered at $(1/2,1/2)^\top$ with diagonal covariance
matrix $\Sigma_f = \diag(\sigma_f^2)$  where $\sigma_f^2 =
(\sigma_{f,1}^2,\sigma_{f,2}^2)$
$$
f(x) := \frac{1}{2 \pi (\det \Sigma_f)^{1/2}} \exp\left( -\frac{1}{2} (x-(1/2,1/2)^\top)^\top
\Sigma_f^{-1} (x - (1/2,1/2)^\top) \right).
$$
We plot $f$ and its regularized potential $V_\tau$ in Figure \ref{fig03}.

\begin{figure}[ht!]
\centering
\begin{tabular}{cc}
\includegraphics[width =.45\textwidth]{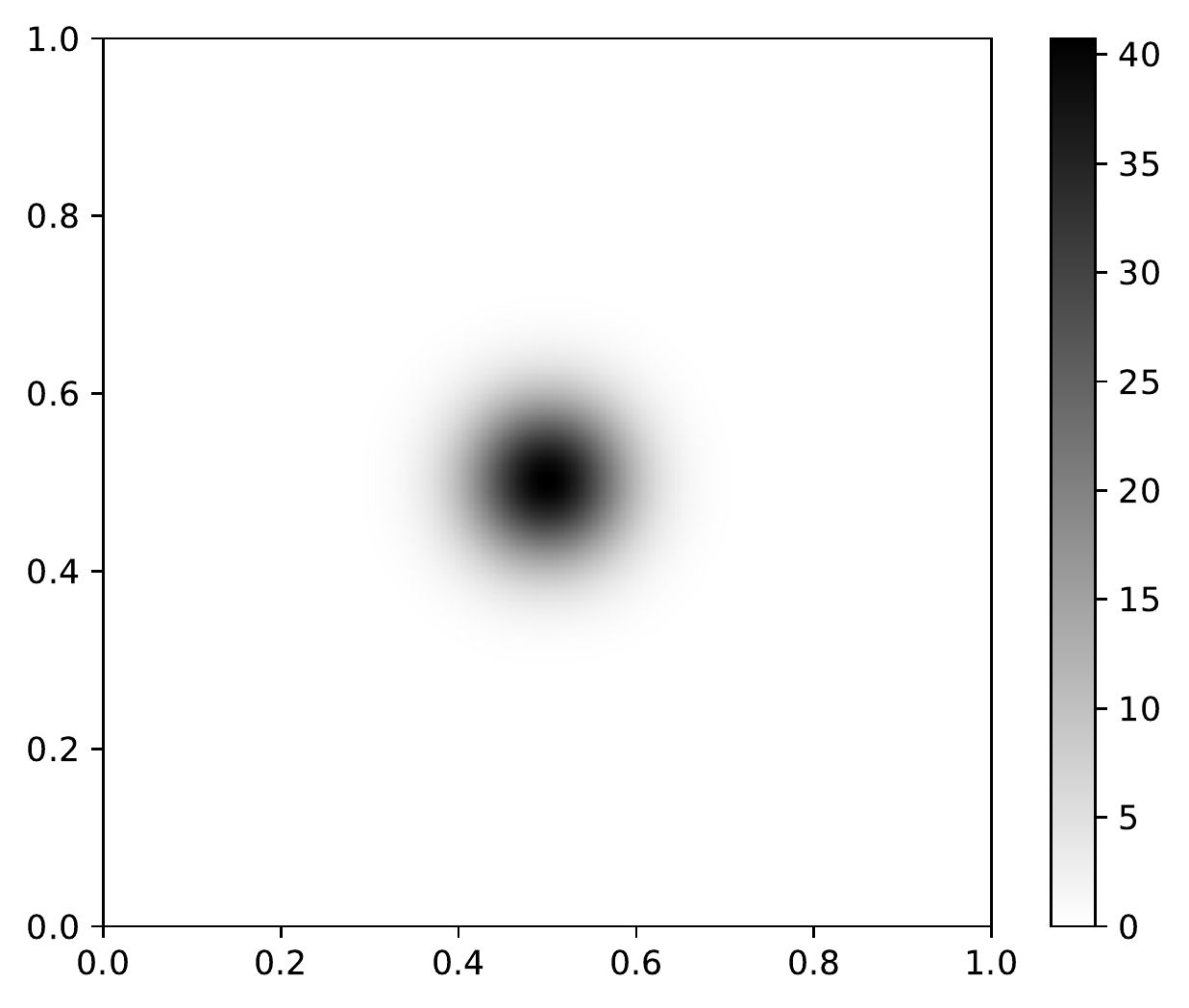}  &
\includegraphics[width =.45\textwidth]{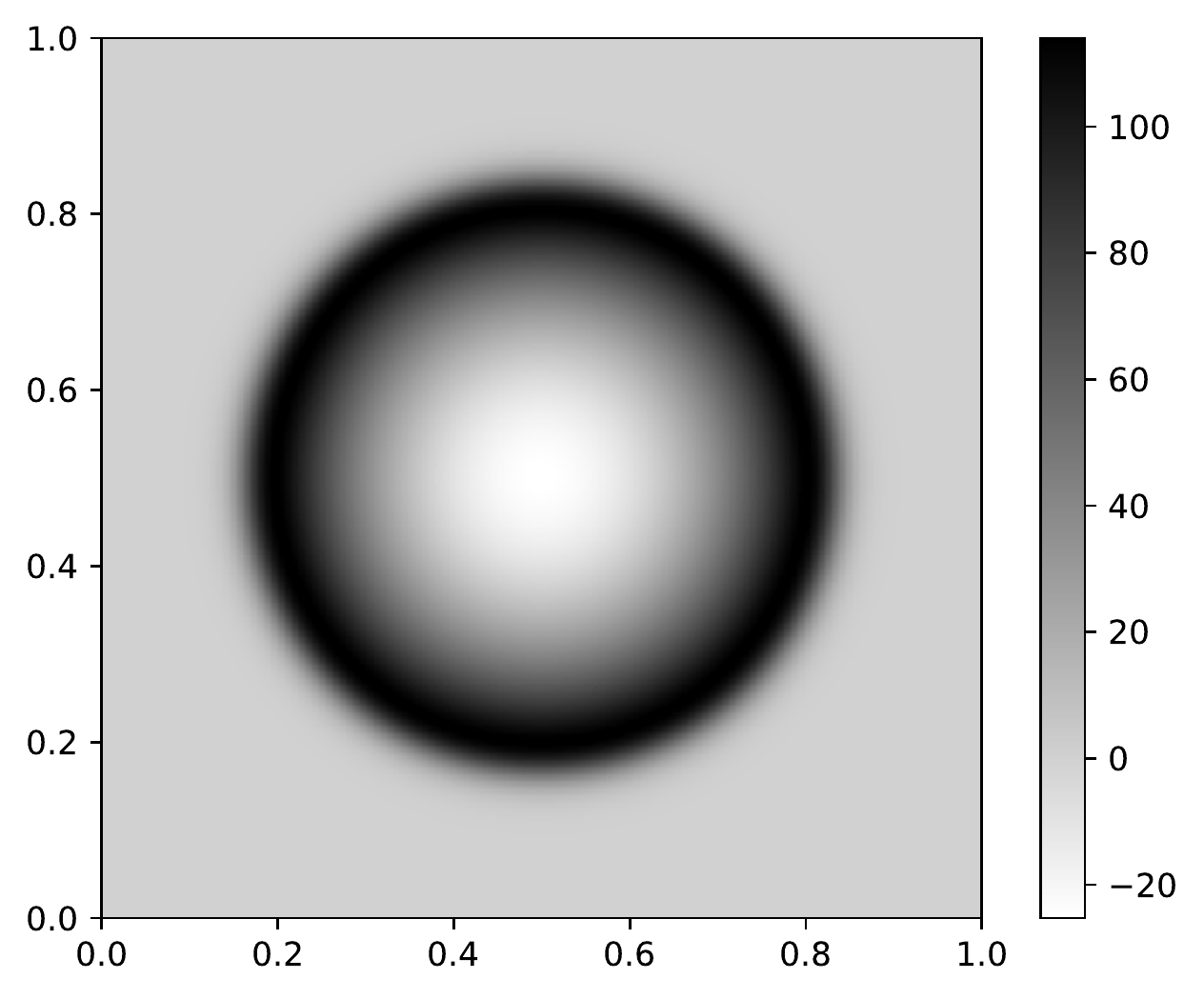}  
\end{tabular}
\caption{The function $f$ (left) and its regularized potential $V_{\tau}$ (right).}
\label{fig03}
\end{figure}

If the density $g$ of $\nu$ is the Gaussian function centered at $(1/2,1/2)^\top$ with
diagonal covariance matrix $\Sigma_g = \diag(\sigma_g^2)$, then recall that by
\eqref{w2gauss} we have
\begin{equation}
\label{w2sigmaa}
W_2(\mu,\nu) = | \sigma_f - \sigma_g |.
\end{equation}
Fix $\varepsilon > 0$, let $S$ denote the unit circle $S := \{ v \in
\mathbb{R}^2 : \|v\|_2 = 1 \}$, and observe that
\begin{equation} \label{w2slicev}
\left\{ W_2(\mu,\mu_{\varepsilon v}) v \in \mathbb{R}^2 : v \in
S \right\} = \varepsilon S,
\end{equation}
where here $\mu_{\varepsilon v}$ is the measure with density $f_{\varepsilon
v}$, where $f_{\varepsilon v}$ is a Gaussian function centered at the $(1/2,1/2)^\top$
with diagonal covariance 
$$
\Sigma_{f_{\varepsilon v}} = \diag( (\sigma_f + \varepsilon v))^2.
$$
That is, $f_{\varepsilon v}$ changes the standard deviations
$\sigma_f$ of the Gaussian $f$ by $\varepsilon v$. The fact that
 \eqref{w2slicev} holds follows from \eqref{w2sigmaa}. As in the previous
section, we study analogs of the set defined in the left hand side of
\eqref{w2slicev}, where the $W_2$ metric is replaced by our linearization, the
unweighted Sobolev norm, and the Euclidean norm, respectively. 
In particular, we define
$$
V_\text{witten} := \left\{ \|(f-f_{\varepsilon
v})/f_\tau\|_{\dot{H}^{-1}(d\mu_\tau)} v \in \mathbb{R}^2 : v \in S \right\},
$$
$$
V_\text{sobolev} := \left\{ \|f - f_{\varepsilon v}\|_{\dot{H}^{-1}(dx)} v \in \mathbb{R}^2
: v \in S \right\},
$$
and
$$
V_\text{euclid} := \left\{ \|f - f_{\varepsilon v}\|_{L^2(dx)} v \in \mathbb{R}^2 : v \in S
\right\}.  
$$
We plot the sets $V_\text{witten}, V_\text{sobolev},$ and $V_\text{euclid}$ in
Figure \ref{fig02}. Note that the set $\varepsilon S$ is included for reference
and is plotted using a dotted line in  the plots of Figure \ref{fig02}.

\begin{figure}[ht!]
\centering
\begin{tabular}{ccc}
\includegraphics[width =.3\textwidth]{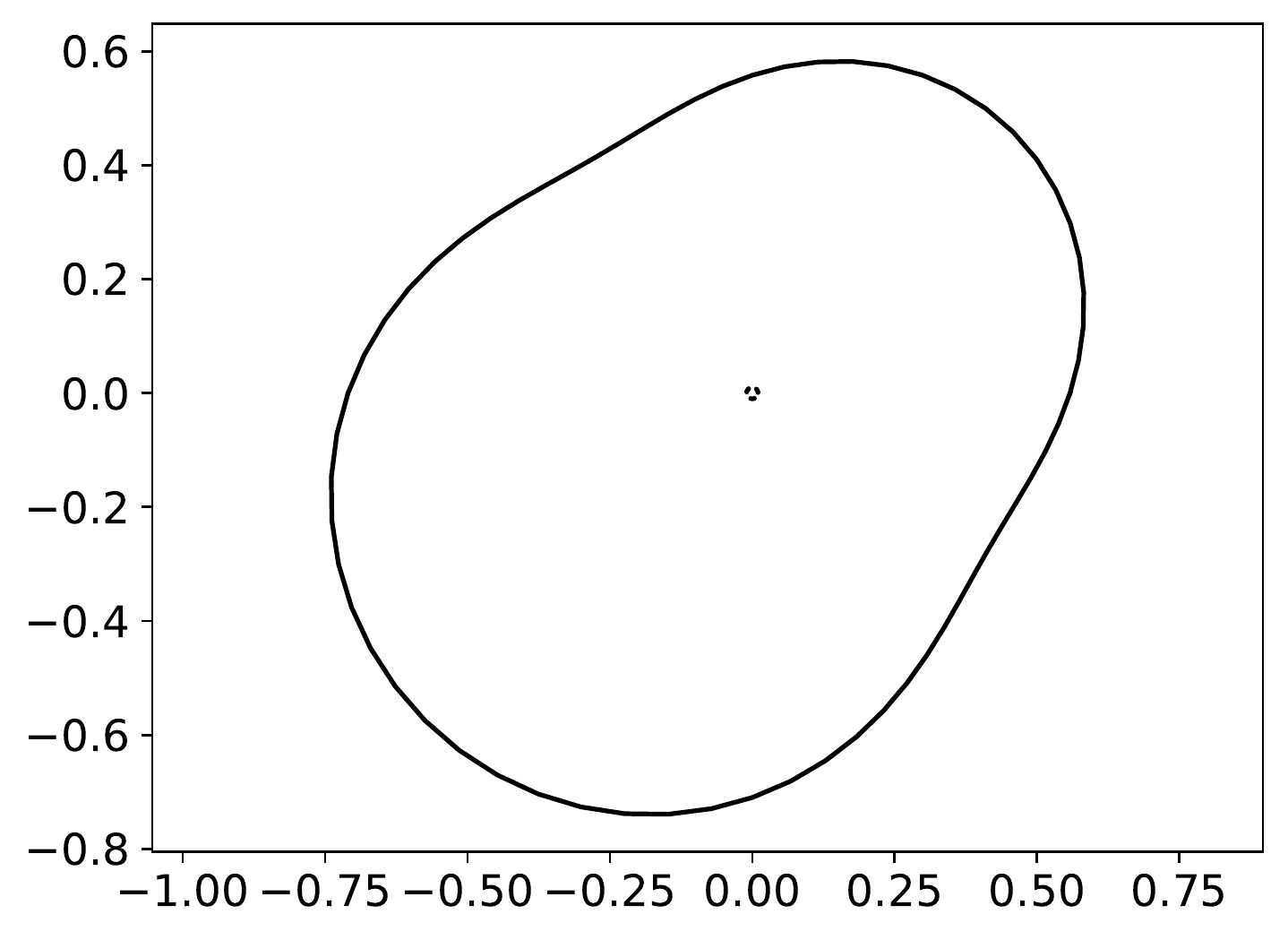} &
\includegraphics[width =.3\textwidth]{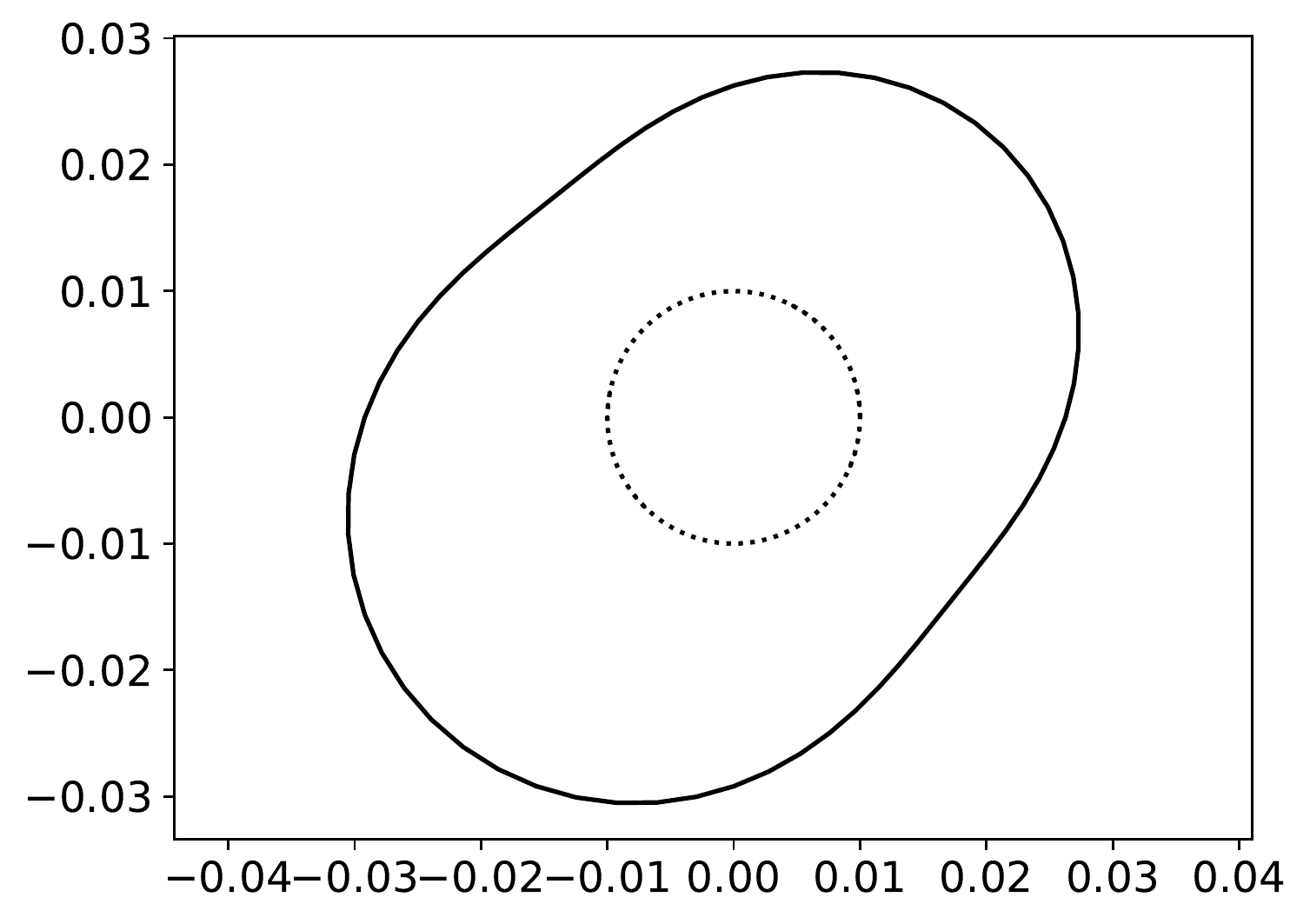} &
\includegraphics[width =.3\textwidth]{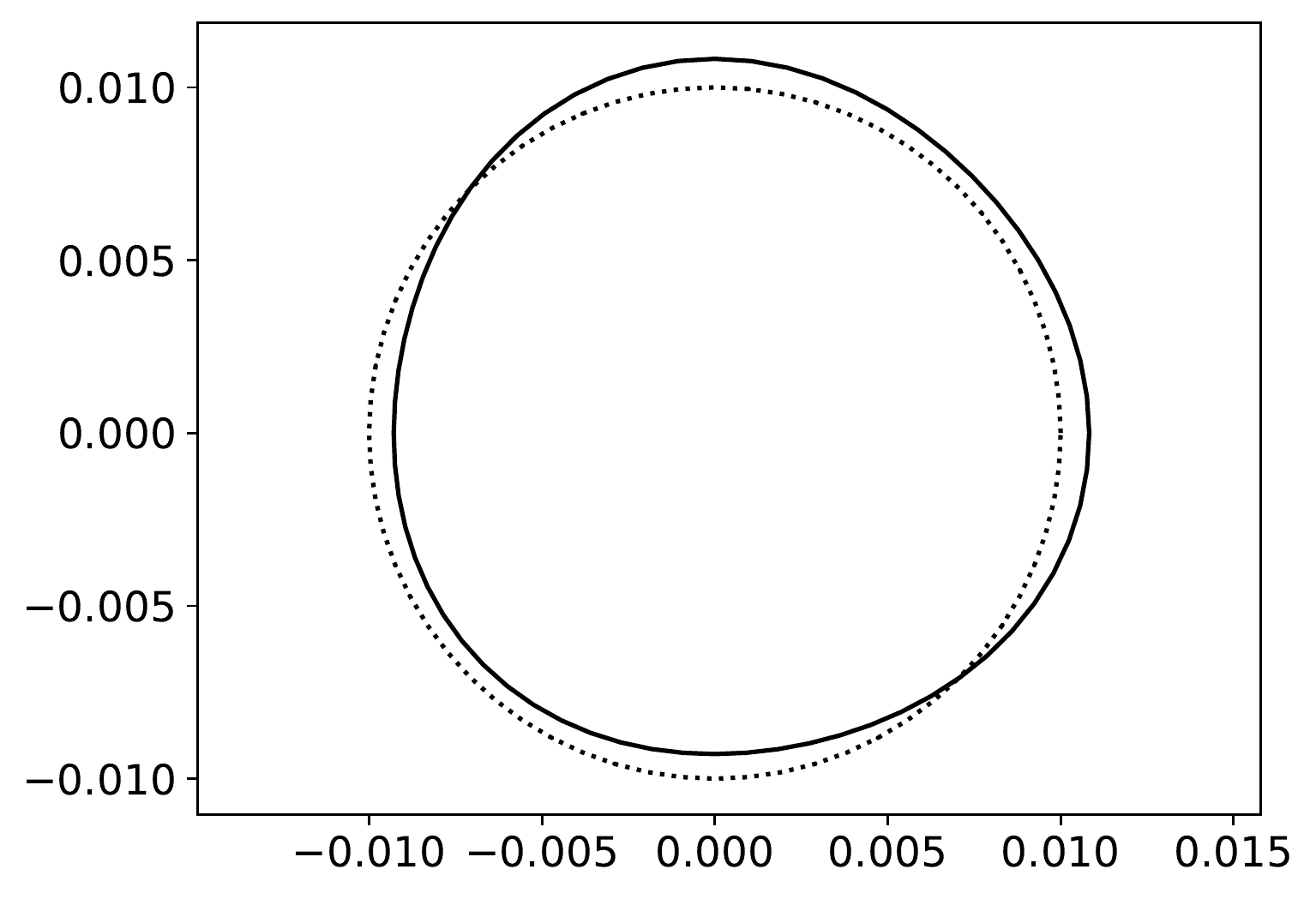}
\end{tabular}
\caption{$V_\text{euclid}$ (left), $V_\text{sobolev}$ (middle), and
$V_\text{witten}$ (right).}
\label{fig04}
\end{figure}

First, observe that the plots of $V_\text{euclid}$ and $V_\text{sobolev}$ in
Figure \ref{fig04} appear stretched in the $(1,1)$  and $(-1,-1)$ directions:
when the vector $\varepsilon v$ changing the standard deviations is in the
positive quadrant this corresponds to increasing the standard deviation of both
variables. Similarly, the negative quadrant (where both components of $v$ are
negative) corresponds to decreasing the standard deviation of both variables.
Both of these deformations result in a similarly large change. In contrast, the
other quadrants (where the components of $v$ have different signs) correspond to
increasing one standard deviation in one direction while decreasing the standard
deviation in the other direction; this explains the asymmetry of
$V_\text{euclid}$ and $V_\text{sobolev}$.  In contrast, the plot of
$V_\text{witten}$ in Figure \ref{fig04} roughly preserves the circle with only a
small deformation.

\subsection{Numerical example: managing noise and computational cost}
\label{costex}

In this section, we remark how the method can be used in more practical
situations involving images. In particular, we consider a $129 \times 129$ image
containing a biomolecule. In order to manage both the computational cost and noise,
we define the potential $V_\tau$ with $\tau = 0.01$, see Figure \ref{fig05}.
\begin{figure}[ht!]
\centering
\begin{tabular}{cc}
\includegraphics[width =.45\textwidth]{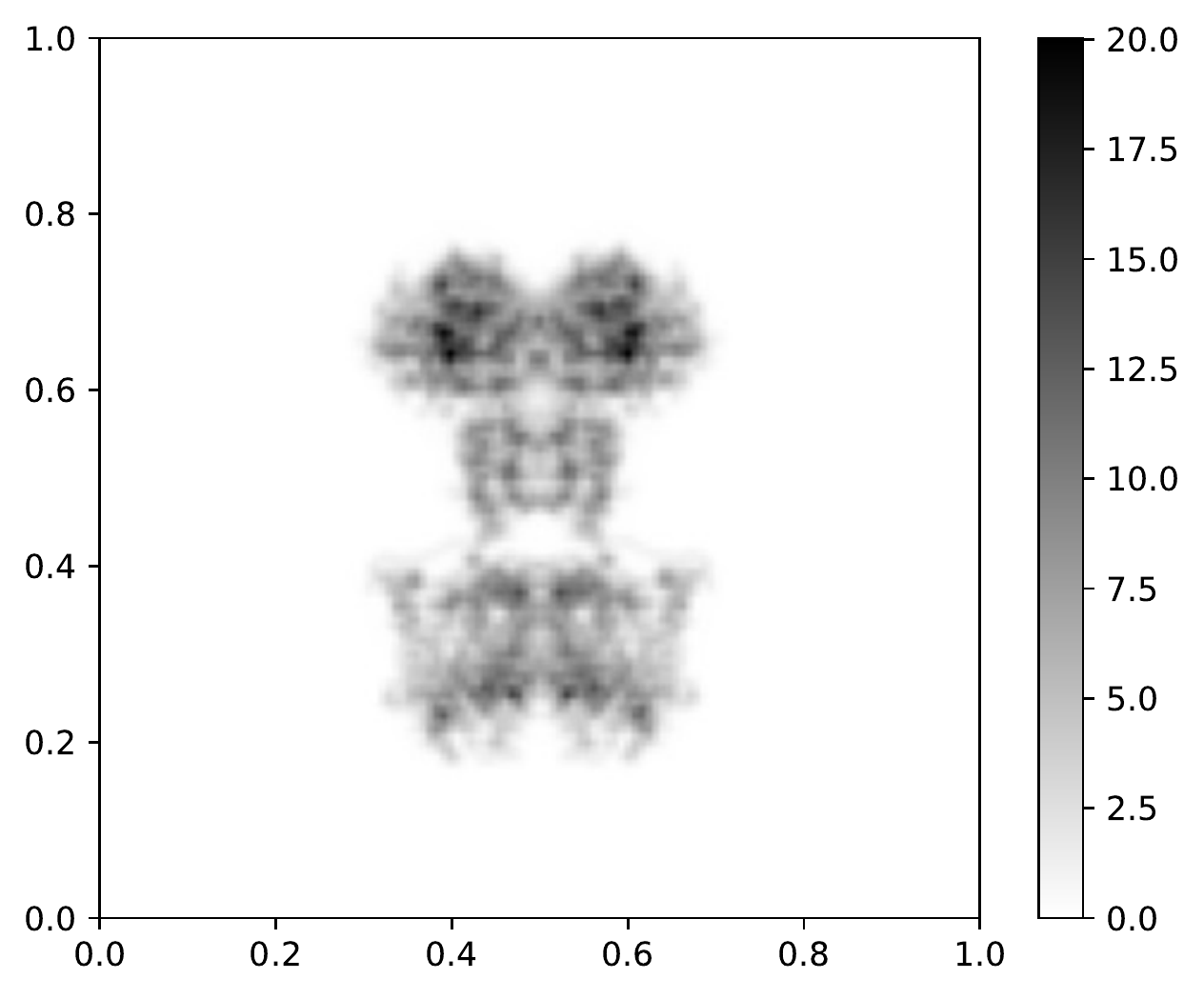}  &
\includegraphics[width =.45\textwidth]{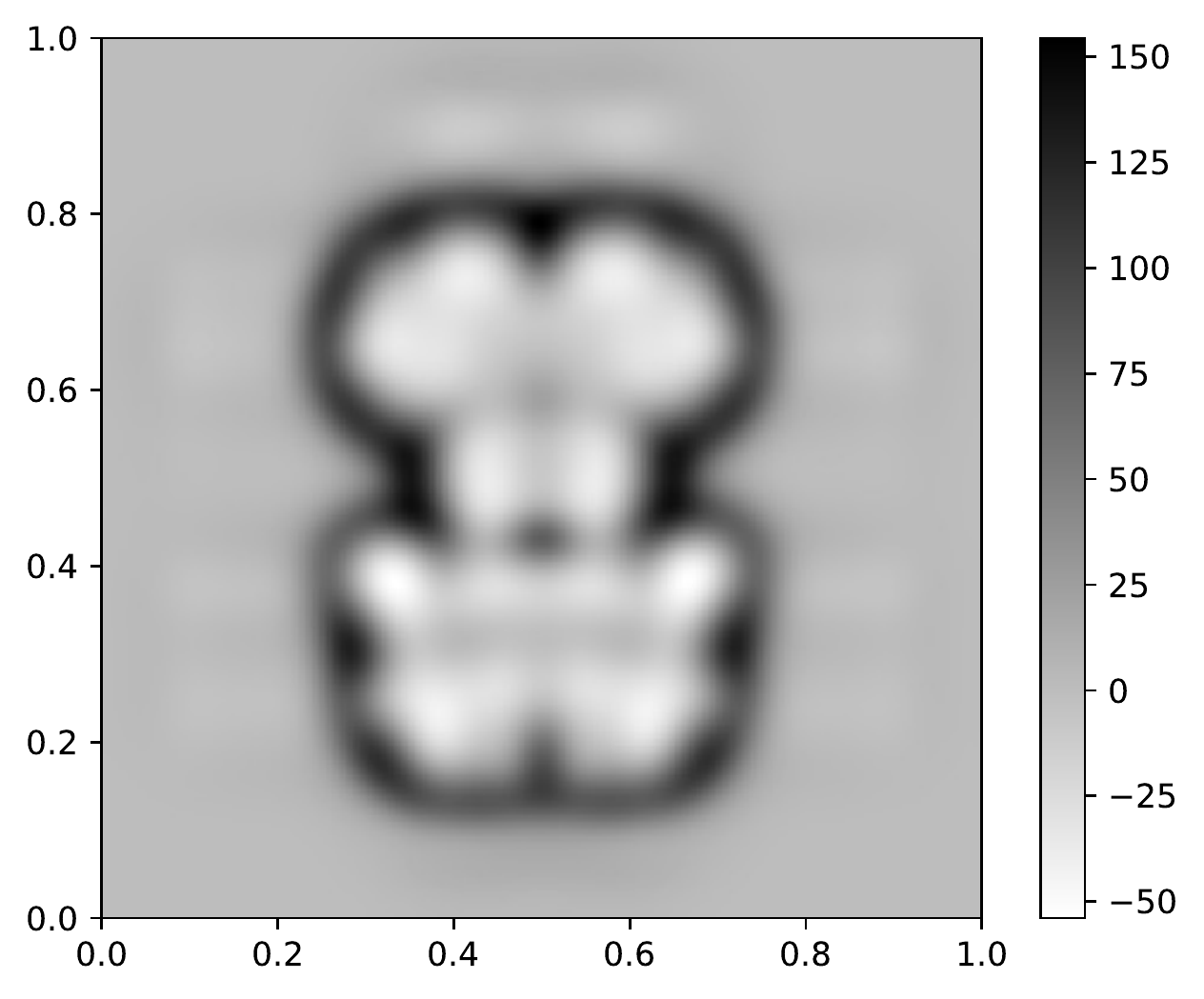}  
\end{tabular}
\caption{Function $f$ (left) and its regularized potential $V$
(right)}
\label{fig05}
\end{figure}
Observe that the maximum value of the potential in Figure \ref{fig05} is about
150. Therefore, we expect the computational cost to be proportional to the
square root of $150$. Using Algorithm \ref{alg1} to compute distances based on
this image has an average time of about $0.035 \text{ (seconds) }$ on a laptop,
where the average is taken over $128$ computations.

\subsection{Numerical example: local embedding} \label{secembed}
In this section, we discuss an immediate extension of the described 
method to defining an embedding of the negative weighted homogeneous Sobolev
norm. In particular, we can define a map
$$
g \mapsto \Phi_f(g) 
$$
such that
\begin{equation} \label{embedH}
\|\Phi_f(g) - \Phi_f(h)\|_{L^2} = \|g - h\|_{\dot{H}^{-1}(d\mu)}.
\end{equation}
Indeed, it follows from \S \ref{wittenf} that if
$$
\Phi_f(g) = H^{-1/2} ((f-g)/\sqrt{f}),
$$
where $H = -\Delta + V$ and the potential $V$ depends on $f$, 
then \eqref{embedH} holds. The operator $H^{-1/2}$ with Neumann boundary
conditions is well defined since $H$ is positive definite on a subspace that
contains $(f - g)/\sqrt{f}$. Computationally, $H^{-1/2}$ can be computed using a
version of $H$ with a regularized potential via
an iterative method based on approximating $\sqrt{x}$ by Chebyshev polynomials. To
illustrate this embedding we define Gaussian functions $f,g,h$ with means
$$
\mu_f = (1/2,1/2), \quad \mu_g = \mu_f + (0.001, 0.002), \quad 
\mu_h = \mu_f + 0.003, -0.002)
$$
and standard deviations
$$
\sigma_f = (1/16, 1/14 ), \quad
\sigma_g = \sigma_f + (0.001, 0.003), \quad
\sigma_h = \sigma_f + (-0.001, 0.002),
$$
respectively, where the Gaussian function is defined by \eqref{gaussss}. 
We compute $\Phi_f(g)$ and $\Phi_f(h)$ and plot the result in Figure \ref{fig06}.
\begin{figure}[ht!]
\centering
\begin{tabular}{cc}
\includegraphics[width =.45\textwidth]{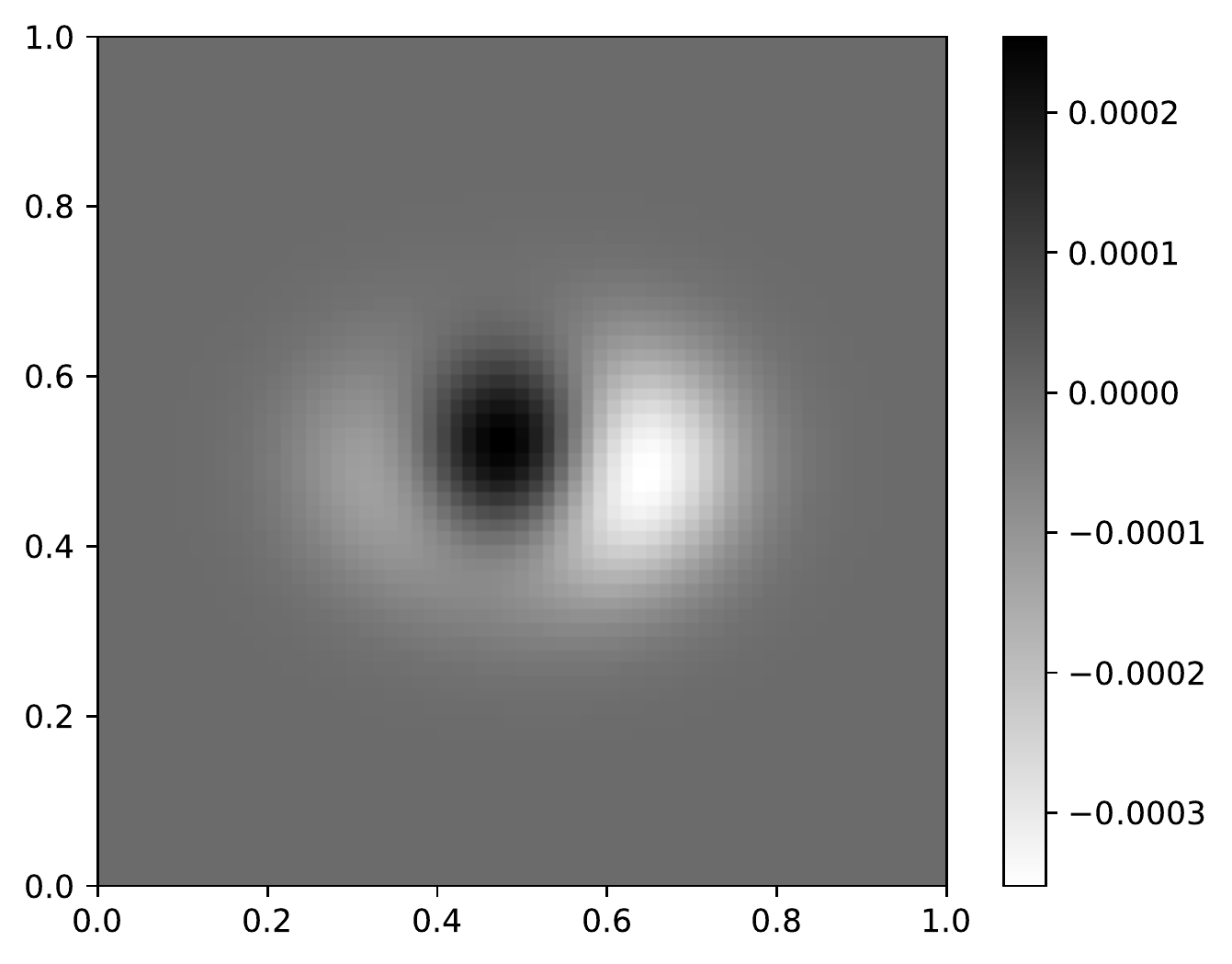}  &
\includegraphics[width =.45\textwidth]{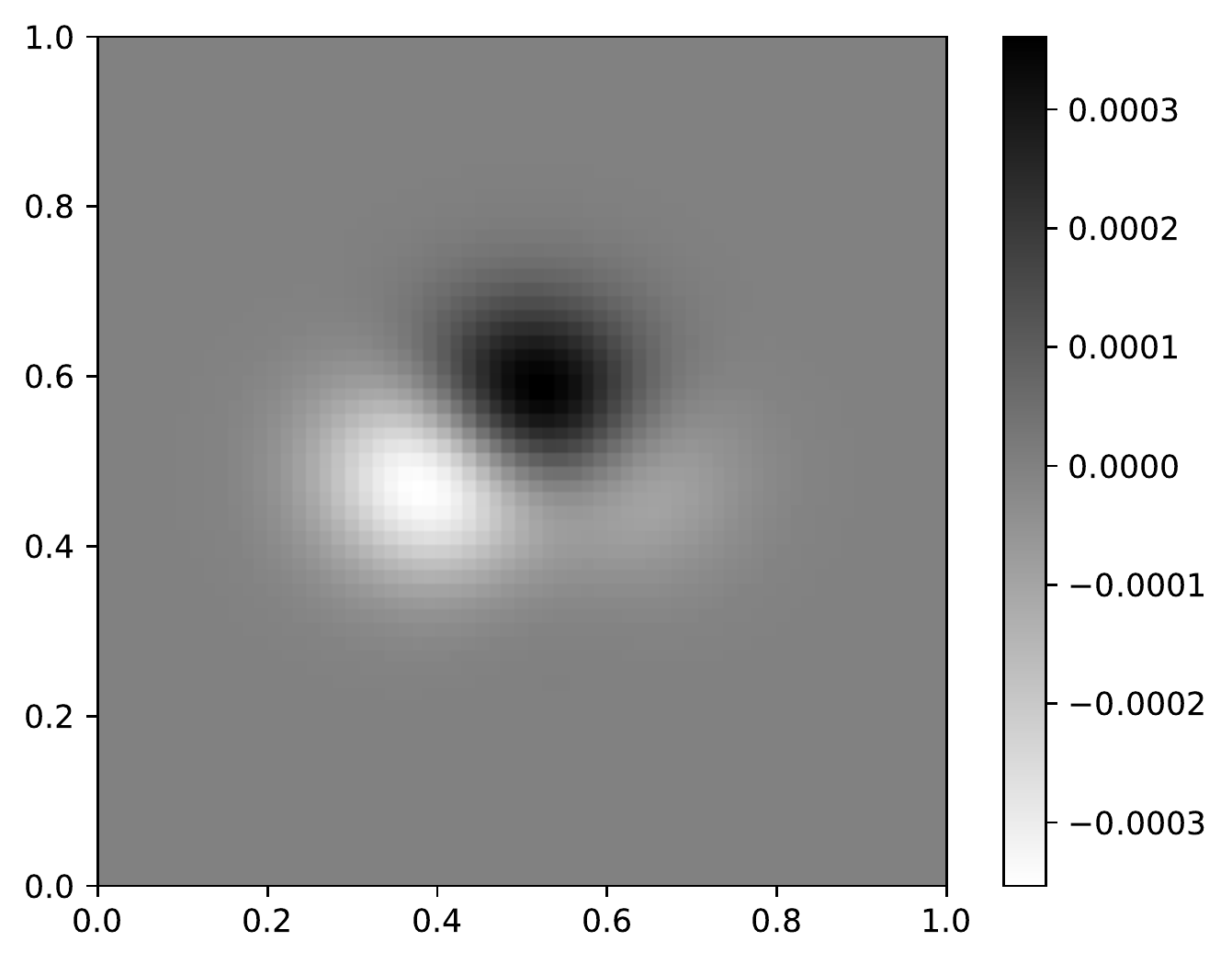}  
\end{tabular}
\caption{$\Phi_f(g)$ (left) and $\Phi_f(h)$
(right).}
\label{fig06}
\end{figure}
We find that
$$
\|\phi_f(g) - \phi_f(h)\|_{L^2} = 4.949 \times 10^{-3},
$$
and using the analytic formula for the $W_2$ metric between these Gaussian
distributions gives
$$
|\|\phi_f(g) - \phi_f(h)\|_{L^2} - W_2(\nu,\eta)| \approx 5.080 \times 10^{-5}
$$
where $\mu$ and $\eta$ are the measures associated with $g$ and $h$,
respectively, which verifies the effectiveness of this embedding as a local 
approximation of the $W_2$ metric.

\section{Discussion} \label{discuss}

In this paper we have studied a classic linearization of Wasserstein distance.
In particular, we focused on the connection between $W_2$ and the Witten
Laplacian which is a Schr\"odinger operator of the form
$$
H = -\Delta + V.
$$
From a computational point of view, the principle advantage of this formulation
is that the computational cost of solving $H = - \Delta + V$ can be roughly
bounded by the square root of the maximum value of the potential (since this
influences the condition number of solving $H \psi = u$ after an appropriate
transformation). The potential $V$ can be smoothed until an acceptable
computational cost is achieved. For example, if the maximum value of $V$ is
$100$, then the number of iterations for Conjugate Gradient will be $\sim 10$, where
each iteration has the same cost of computing the unweighted Sobolev norm. The
numerical experiments indicate how this Witten Laplacian distance will,
roughly speaking, preserve the Wasserstein distance for small balls around
measures.  This perspective opens the possibility of many interesting
applications. For example, the operator $H$ can be used to smooth images via the
diffusion 
$$
f \mapsto \exp (-\tau H) f,
$$
where $\exp$ denotes the operator exponential, which is interesting since the
infinitesimal generator $H$ of the diffusion has a connection to Wasserstein
distance. More generally, the fractional diffusion
$$
f \mapsto \exp (-\tau H^\alpha) f,
$$
for $\alpha > 0$
can be considered where $H^\alpha$ is the operator to power $\alpha$, which is
well defined since $H$ is positive semi definite. There are other interesting
applications about defining embeddings into Euclidean space, and potential
applications to graphs. Another potential application is related to maximum
mean discrepancy (MMD), which has been consider by other authors, see the
discussion in \S \ref{related}, but our perspective may offer some new ideas.

\end{document}